\documentclass[a4paper,11pt]{amsart}
\usepackage{amsmath,amssymb,amsthm} 
\usepackage{graphics}
\usepackage{comment}
\usepackage{enumitem}
\usepackage{tikz}
\usetikzlibrary{calc}
\usetikzlibrary{patterns}
\usetikzlibrary{spy}
\usetikzlibrary{decorations.pathreplacing}

\numberwithin{equation}{section}

\newtheorem{theorem}{Theorem}[section]

\newtheorem{proposition}[theorem]{Proposition}

\theoremstyle{definition}
\newtheorem{definition}[theorem]{Definition}
\newtheorem{remark}[theorem]{Remark}

\theoremstyle{plain}
\newtheorem*{claim*}{Claim}

\newcommand{\R}{\mathbb{R}}
\newcommand{\N}{\mathbb{N}}

\newcommand{\vhi}{\varphi}
\newcommand{\eps}{\varepsilon}

\newcommand{\lt}{\left}
\newcommand{\rt}{\right}

\begin{document}

\title[Concentration-compactness for general isoperimetric problems]{A concentration-compactness principle for perturbed isoperimetric problems with general assumptions}
\author[J. Candau-Tilh]{Jules Candau-Tilh}
\address{J. C-T.: Univ. Lille, CNRS, UMR 8524, Inria - Laboratoire Paul Painlev\'e, F-59000 Lille}
\email{jules.candautilh@univ-lille.fr}

	\begin{abstract}	
	Derived from the concentration-compactness principle, the concept of generalized minimizer can be used to define generalized solutions of variational problems which may have components ``infinitely" distant from each other. In this article and under mild assumptions we establish existence and density estimates of generalized minimizers of perturbed isoperimetric problems. Our hypotheses encapsulate a wide class of functionals including the classical, anisotropic and fractional perimeter. The perturbation term may for instance take the form of a potential, a translation invariant kernel or a nonlocal term involving the Wasserstein distance.
	\medskip
	
	\noindent
	\textbf{Keywords and phrases.} Nonlocal isoperimetric problems, concentration-compactness principle, generalized minimizers.
	\medskip
	
	\noindent
	\textbf{2020 Mathematics Subject Classification.} 49J45, 49Q10, 49Q20.
	\end{abstract}
\maketitle

\section{Introduction}

One of the first perturbed isoperimetric problems was formulated by George Gamow in the 1930s for investigating the stability of the atomic nucleus \cite{Gam30}. Given $d \geq 2$, $m>0$ and $\alpha \in (0, d)$ a possible formulation of this variational problem is
\begin{equation*}
\inf_{|E|=m} \bigg\{\mathrm{Per}(E) + \int_{E \times E} \frac{dxdy}{|x-y|^{d-\alpha}}\bigg\},
\end{equation*}
where $\mathrm{Per}$ is the Caccioppoli perimeter and $|E|$ is the Lebesgue measure of $E$. The goal of this variational problem is to model an attractive, short-range force inducing surface tension (the ``perimeter" term) that competes with a repulsive term $V$ acting at a greater distance (the ``perturbation" term, which is often nonlocal). This competition plays a pivotal role in the wide range of geometries the perturbed isoperimetric problem can describe (see for instance \cite[Figure 1]{KnMuNov16}) and both the physics and mathematics communities have explored numerous variants of this problem. In this article we study a generalized version of this problem:
\begin{equation}\label{infclass}
	e(m) = \inf_{|E|=m} \bigg\{ \mathcal{E}(E) = P(E) + V(E) \bigg\},
\end{equation}
where $P$ is a non-explicit, nonnegative perimeter term and $V$ a perturbation term. 

Mathematically speaking, a primary concern in tackling this optimization problem is establishing the existence of solutions. The most challenging task  often is to exhibit convergent minimizing sequences. From a concentration-compactness principle standpoint \cite{Lio84}, lack of compactness in isoperimetric problems can occur when a minimizing sequence admits components fleeing infinitely far apart. However, in such scenarios it may still be possible to show that some relaxed versions of the problem admit minimizers. It is within this framework that we introduce the generalized minimization problem:
\begin{equation}\label{infgen}
e_{\mathrm{gen}}(m) = \inf_{} \Bigg\{\mathcal{E}_{\mathrm{gen}}((E^i)_{i\geq1}) = \sum_{i\geq 1} \mathcal{E}(E^i): \sum_{i\geq 1} |E^i| = m \Bigg\}.
\end{equation}
In the context  of metric measure spaces and with $V=0$, research towards finding minimal assumptions guaranteeing existence of solutions to \eqref{infgen} was carried out in \cite{NovPaoStepTor21}, where existence of generalized isoperimetric clusters was also shown. See also \cite{AntNarPoz22} for a characterization of minimizing sequences for the isoperimetric problem on noncompact $RCD(K,N)$ spaces.  

In addition to the matter of their existence, the question of the regularity of minimizers constitutes a significant aspect of the study of isoperimetric problems. It is indeed well-known in shape optimization theory that studying variational problems in the class of sets whose boundary has some regularity allows for much easier computations and characterization of the solutions. A first step towards establishing regularity properties of minimizers is often to prove that, when they exist, they have density estimates. We say that a set $E$ admits interior (resp. exterior) density estimates when there exists $c, r'>0$ such that for any $0 < r \leq r'$ and $x \in E$ (resp. $E^c$),
\[
 |E\cap B_r(x)|\geq c \, r^d \, (\text {resp. } |E^c\cap B_r(x)| \geq c \, r^d \, ).
\]
It can also be useful to show density estimates for sets that are close (for a given topology) to minimizers of a given isoperimetric problem. Indeed, it is then often possible to modify a bit those sets to obtain actual minimizers of the considered problem. See for instance results obtained in \cite{DeRNeu23} in the context of Riemannian manifolds regarding the regularity of volume-constrained local minimizers of anisotropic surface energies.

Our present goal is to exhibit in the case $V \not = 0$ general assumptions under which:
\begin{itemize}
\item[$-$] \eqref{infclass} and \eqref{infgen} coincide, 
\item[$-$] \eqref{infgen} admits solutions,
\item[$-$] solutions of \eqref{infgen} have density estimates.
 \end{itemize}

\subsection{Main results}\hfill

Let us denote by $(e_i)_{i=1}^d$ the canonical basis of $\R^d$ and specify that all the sets considered in the article are assumed to be at least (Lebesgue) measurable. We start by proving in Section \ref{exisgen} that \eqref{infclass} and \eqref{infgen} coincide under the following set of assumptions (S1). \medskip

\begin{enumerate}[label=(H\arabic*), itemsep=4pt]
	\item\label{vanloc} \textit{Energy of small balls}: $\mathcal{E}(B_r)\to 0$ as $r \to 0$ and $\mathcal{E}(\emptyset) = 0$.
	\item\label{vaninf} \textit{Convergence at infinity}: For any set $E$ with $|E|< \infty$, $\mathcal{E}(E \cap B_R) \to \mathcal{E}(E)$ as $R \to \infty$.
	\item\label{decker} \textit{Vanishing range of action}: If $E,F$ are bounded sets, then $\mathcal{E}(E\cup(F+Le_1))\to\mathcal{E}(E)+\mathcal{E}(F)$ as $L \to \infty$.
\end{enumerate}

\medskip

\begin{proposition}\label{infclassinfgen}
Assume that $\mathcal{E}$ satisfies (S1). Then $\eqref{infclass}=\eqref{infgen}$.
\end{proposition}

Let us comment a bit on (S1). We use \ref{vanloc} to compensate for any potential mass deficit when we modify a set $E$ to construct a generalized minimizer $(E^i)_i$. However, there are alternative methods to ensure that the mass constraint is satisfied when solving \eqref{infclass} or \eqref{infgen} (see e.g. Remark \ref{scaling}). \ref{decker} states that bounded sets do not interact when infinitely far apart from each other, so that they may be seen as components of a generalized minimizer.\medskip

We then show that \eqref{infgen} admits minimizers. To prove this result we introduce the functionals $E \mapsto P(E, U)$ and $E \mapsto V(E,U)$, which are defined relatively to a Lebesgue measurable set $U$. By convention, we write $P(E, \R^d) = P(E)$ and $V(E, \R^d) = V(E)$. The set of assumptions (S2) we require to establish that \eqref{infgen} has solutions is as follows (we use the letter $\mathcal{F}$ to denote $P$ or $V$ in hypotheses applying to both terms):\medskip

\begin{enumerate}[resume*, label=(H\arabic*)]
	\item\label{reliso} \textit{Relative isoperimetric inequality}: There exists $r_0>0$ and $ f_1 : \mathbb{R}_+ \to \mathbb{R}_+$ increasing with $f_1(0)=0$, $m \mapsto f_1(m)/m$ nonincreasing and $\lim_{m \to 0} f_1(m)/m = \infty$ such that for $r \leq r_0$, $x \in\R^d$ and $E \subset \R^d$:
	\begin{equation*}
		\min \big(f_1(|E \cap B_r(x)|), f_1(| B_r(x) \setminus E|)\big) \leq P(E, B_r(x)).
	\end{equation*}
	\item\label{trainvloc} \textit{Periodicity}: There exists $0 < r_1 \leq 2 r_0/\sqrt{d}$ such that $\mathcal{F}(E + r_1 , U + r_1 e_k) = \mathcal{F}(E, U)$ for any $1 \leq k \leq d$ and $E, U \subset \R^d$.
	\item\label{Psetop} \textit{Perimeter set operations}: Given $E \subset \R^d$, if $U \subset U'$ are open, then $P(E, U) \leq P(E, U')$, and if $(U_i)_{i=1}^I$ are open disjoint sets, then $\sum_{i=1}^I P(E, U_i) \leq P(E, \cup_{i=1}^I U_i)$.
	\item\label{bdEbdP} \textit{Bounded perimeter}: If $(E_n)_n$ is such that $\sup_n \mathcal{E}(E_n) < \infty$ and $\sup_n |E_n| < \infty$, then $\sup_n P(E_n) < \infty$. 
	\item\label{com} \textit{Compactness}: If $(E_n)_n$ satisfies $\sup_n P(E_n, U) < \infty$, then up to extraction there exists $E \subset \R^d$ such that $ E_n \cap U \to E$ in $L^{1}_{\mathrm{loc}}$ as $n \to \infty$.
	\item\label{semconloc}  \textit{Lower semicontinuity}: Given $E \subset \R^d$ and  a bounded open set $U$,  if $E_n \to E$ in $L^1_{\mathrm{loc}}$ as $n \to \infty$ with $\sup_n |E_n| < \infty$, then $\mathcal{F}(E, U) \leq \liminf_n \mathcal{F}(E_n, U)$.
	\item\label{bep} \textit{Beppo-Levi}: If $(U_n)_{n\geq 0}$ is a nondecreasing sequence of open sets exhausting $\R^d$, then for any $E \subset \R^d$, we have $\mathcal{F}(E, U_n) \to \mathcal{F}(E)$ as $n \to \infty$.
	\item\label{wksupadd} \textit{Weak superadditivity}: For any $m>0$ there exists $\eta_1, \eta_2: \R_+ \to \R_+$ with  $\eta_1$ continuous, $\eta_1(0)=0$ and $\eta_2(r) \to 0$ as $r\to\infty$ such that the following holds: for any $E\subset \R^d$ with $|E|\leq m$ and any finite family of balls $(B^i)_{i=1}^I$ of radius $R>0$ such that $\min_{i\neq j}\mathrm{dist}(B^i, B^j) \geq 5R$,
	\[
	\sum_{i=1}^I V(E, B^i) \leq V(E) + \eta_1\Big(\big|E \setminus \cup_{i=1}^I B_i\big|\Big) + \eta_2(R).
	\] 
\end{enumerate}

\medskip

\begin{theorem}\label{infgenmingen}
	Assume that the relative functionals of $P$ and $V$ satisfy (S2) and that \eqref{infclass} and \eqref{infgen} coincide. Then, \eqref{infgen} admits a solution.
\end{theorem}

These assumptions deserve some comments. In terms of the concentration-compactness principle, the relative isoperimetric inequality \ref{reliso} allows us to exclude the ``vanishing" case. \ref{trainvloc} is a weakened form of the invariance by translation. \ref{bdEbdP} is trivial when the perturbation term is nonnegative. Let us point out that in \cite{NovPaoStepTor21} the authors establish existence of isoperimetric clusters in homogeneous metric spaces and with $V=0$. In particular, their results imply existence of isoperimetric sets in $\R^d$. A comparison between (S2) and their set of hypotheses reveals that they are essentially identical. Indeed when $V=0$, hypotheses \ref{bdEbdP} and \ref{wksupadd} are superfluous and (S2) is analogous to the hypotheses of \cite[Section 2 \& Theorem 3.3]{NovPaoStepTor21}. Points $(i)$ and $(ii)$ of \cite[Theorem 3.3]{NovPaoStepTor21} are obtained in our case through the partition of $\R^d$ into cubes.

In the first part of Section \ref{densest}, we show that $\rho$-minimizers of the perimeter (see \cite[Section 21]{Maggi12} for the related concept of $(\Lambda, r_0)$-minimizers of the perimeter) have interior and exterior density estimates under the set of hypotheses (S3).
\begin{definition}\label{quasimin}
	Let $\rho : \mathbb{R}_+ \to \mathbb{R}_+$ be nondecreasing. We say that $E \subset \R^d$ is a $\rho$-minimizer of the perimeter (or simply  a $\rho$-minimizer) if there exists $r_2>0$ such that for any $r \leq r_2$, $x\in\mathbb{R}^d$ and $E' \subset \R^d$ with $E \Delta E'\subset B_r(x)$ we have
	\begin{equation}\label{eqquasimin}
		P(E) \leq P(E') + \rho(r).
	\end{equation} 
	The function $\rho$ is called the error function for $E$.
\end{definition}
\noindent
The set of assumptions (S3) is made of \ref{reliso} (relative isoperimetric inequality) and \ref{Psetop} (set operations) as well as three new hypotheses:\medskip
\begin{enumerate}[resume*, label=(H\arabic*)]
	\item\label{loccom} \textit{Local comparisons}: For $E \subset \R^d$, $x \in \R^d$ and a.e. $r>0$,
	\[
	P(E) = P(E, B_r(x)) + P(E, B_r(x)^c).
	\]
	Additionally, for some $C>0$:
	\begin{align*}
	P(E \setminus B_r(x)) &\leq P(E, B_r(x)^c) + CP(B_r(x), E),\\
	P(E \cup B_r(x)) &\leq P(B_r(x), E^c) + P(E) - C P(E^c, B_r(x)).
	\end{align*}
	
	\item\label{intdif} \textit{Integral inequality}: There exists $ \, f_2: \mathbb{R}_+ \times \mathbb{R}_+ \to \mathbb{R}_+$  and $r_2$ such that for any $E\subset \R^d$, $x\in\R^d$ and $0 < r \leq r_2$
	\[
	\frac{1}{r}\int_0^r P(B_s(x), E) ds \leq  f_2(r, |E \cap B_r(x)|).
	\]
	\item\label{cro} \textit{Density scale factor}: Let $f_1$ and $f_2$ be given by \ref{reliso} and \ref{intdif} respectively and define
	\[
	f_3(r, m) = 2^d\lt(\frac{f_2(r, m)+\rho(r)}{f_1(m)}\rt) \quad \text{for}\quad r,m>0.
	\]
	Then, there exist $r_3, \eps_1>0$ such that
	\[
	f_3(r,m) \leq 1 \; \; \text{for every}\; \; r \leq r_3 \; \; \text{and every} \; \; \frac{\eps_1}{2^d} r^d < m \leq \eps_1 r^d. 
	\]
\end{enumerate}
\begin{theorem}\label{densityestimates}
	Let $E \subset \R^d $ be a $\rho$-minimizer of the perimeter for some error function $\rho$. If (S3) holds, then there exists $C_0, r_4>0$ such that for $r \leq r_4$,
	\begin{equation}\label{intdens}
		|E \cap B_r(x) | \geq C_0 r^d \text{ for every } x \in E^{(1)}
	\end{equation}
	and
	\begin{equation}\label{outdens}
		|B_r(x) \setminus E| \geq C_0 r^d \text{ for every } x \in E^{(0)},
	\end{equation}
	where  for $t \in [0,1]$, $E^{(t)}$ denotes the points of density $t$ of $E$.
\end{theorem}
\noindent
Let us provide some context on (S3) and Theorem \ref{densityestimates}. Density estimates for $\rho$-minimizers of the perimeter are often an important tool in the study of isoperimetric problems. Indeed, it can be used to show that said $\rho$-minimizers have bounded connected components. Additionally, it is usually a crucial first step in the study of the spherical excess of minimizers (see \cite[Section 22 \& 26]{Maggi12}), a central concept of the regularity theory for minimizers of isoperimetric problems. For illustrations of this concept, one can refer to \cite{GolNovRuf22} for the anisotropic perimeter and to \cite{CafRoqSav} for the fractional perimeter. Additionally, when the density estimates are independent of the considered $\rho$-minimizer, it may be possible to exhibit minimizers for the classical problem \eqref{infclass} provided additional assumptions on the perturbation term $V$ (such as finite or rapidly decreasing range of action). See \cite{Rig2000} for an illustration of this approach with $P$ the classical perimeter and $V$ a nonlocal kernel. Refer also to \cite{CesaNov17} for an example where $P$ is the fractional perimeter and $V$ is the integral of a periodic function and is not necessarily positive, or to \cite{CanGol22} for the case of $P = \mathrm{Per}$ and $V$ is defined using the Wasserstein distance. \medskip

The methodology employed to establish this kind of theorems is now well understood. Since the publication of De Giorgi's seminal papers on the classical isoperimetric problem in the 1950s, various strategies been developed to address isoperimetric problems where the considered perimeter is anisotropic or nonlocal, or with different perturbation terms. However, most of these proofs revolve around the same idea: apply the relative isoperimetric inequality to $E$ (resp. $E^c)$ and integrate this inequality to obtain interior (resp. exterior) density estimates (see \cite[Remark 15.16]{Maggi12}). Consequently, we have aimed at formulating streamlined hypotheses to encompass this shared framework, and also to simplify the process of establishing density estimates in future research on isoperimetric problems. In our framework, we need \ref{loccom} to deduce local results from the $\rho$-minimality of a set $E$, which is a priori a global property. \ref{intdif} is used together with the relative isoperimetric inequality to allow us to compare perimeters and Lebesgue measures. Finally, \ref{cro} ensures that the error function $\rho$ of a $\rho$-minimizer $E$ is a perturbation of higher order of $P(E)$. 

\begin{remark}
	The conditions on $\rho$ specified in \ref{cro} are mild enough that $\rho$-minimizers of the anisotropic perimeter $P_{\phi}$ (resp. of the fractional perimeter $P_s$) have density estimates in the following two cases:
	\begin{itemize}
	\item[$-$] $\rho(r) = C r^{d-1 + \alpha}$ (resp. $\rho(r) = C r^{d-s+\alpha}$) with $\alpha \in (0,1)$ and any $C >0$,
	\item[$-$] $\rho(r) = C r^{d-1}$ (resp. $\rho(r) = C r^{d-s}$) and $C$ small enough.
	\end{itemize}
	Theorem \ref{densityestimates} is thus in accordance with \cite[Proposition 3.1]{GolNovRuf22} and \cite[Theorem 5.7]{CinPra17}.
\end{remark}

In the second part of Section \ref{densest}, we establish the connection between generalized minimizers and $\rho$-minimizers. We prove that generalized minimizers of \eqref{infgen} are $\rho$-minimizers of the perimeter for some $\rho$ in two different situations: a case where $P$ admits volume-fixing variations and a case where both $P$ and $V$ have a scaling property.

\begin{definition}\label{volfix}
	Let $E \subset \R^d$ be such that $P(E) + V(E) < \infty$. We say that $E$ admits volume fixing variations if there exist  $g_1,g_2 : \mathbb{R}_+ \to \mathbb{R}_+$ nondecreasing and $r_5, \eps_2 >0$ with the following properties : if $E' \subset \R^d$ is such that $E \Delta E' \subset B_{r_5}(x)$ for some $x \in \R^d$, then
	\begin{enumerate}
		\item for any $\eps$ such that $|\eps|< \eps_2$, there exist $F \subset \mathbb{R}^d$ and $x_0 \in \R^d$ such that $B_{r_5}(x)$  and  $B_{r_5}(x_0)$ are disjoint and
		\begin{equation*}
			F \Delta E \subset B_{r_5}(x_0), \quad |F|-|E| = \eps, \quad \mathcal{E}(F) \leq \mathcal{E}(E) + g_1(|\eps|).
		\end{equation*}
		\item If $F, F'\subset \R^d$ are such that $E \Delta E' = F \Delta F'  \subset B_{r}(x)$ for $r \leq r_5$ and $E \Delta  F = E' \Delta F' \subset B_{r_5}(x_0)$ with $B_{r_5}(x)$ and $B_{r_5}(x_0)$ disjoint, then
		\begin{equation}
			P(F') - P(F)  \leq P(E') - P(E) + g_2(r).
		\end{equation}
	\end{enumerate}
\end{definition}

\medskip
\noindent
We now introduce the following set of hypotheses, denoted (S4):\medskip
\begin{enumerate}[resume*, label=(H\arabic*)]
	\item\label{sca} \textit{Scaling}: If $E$ minimizes \eqref{infclass}, then there exists $ \alpha, \beta \in \R$ and $t_0>0$  such that
	\begin{equation}\label{scalhyp}
		P(tE) \leq t^{\alpha} P(E) \quad \text{and} \quad V(tE) \leq t^{\beta} V(E) \quad \text{for any } t \text{ such that } |t - 1|\leq t_0.
	\end{equation}
	Additionally there exists $\delta \in [0,1]$, $\gamma \geq 0$ and $C_1 >0$ (if $\delta=0$ we require $0 < C_1 < 1$) such that for any $E \subset \R^d$,
	\begin{equation}\label{addhypquantV}
	V(E) \geq -C_1 |E|^{\delta} P(E)^{1-\delta}.
	\end{equation}
	\item\label{eps} \textit{Volume-fixing variations}:  If $E$ solves \eqref{infclass}, then $E$ admits volume-fixing variations.	
	\item\label{loc} \textit{Local perturbation control}: There exists $\, v : \R_+ \to \R_+$ nondecreasing and $r_6 > 0$ such that for $r \leq r_6$, if $E, E' \subset \R^d$ satisfy $E \Delta E' \subset B_r(x)$ for some $x \in \R^d$, then
	\[
	\lt| V(E) - V(E') \rt| \leq  v(r).
	\]
\end{enumerate}

\begin{proposition}\label{mingentoquasimin}
Assume that the relative functionals of $P$ and $V$ satisfy (S4) with either \ref{sca} or \ref{eps}. Then every component of a generalized minimizer of \eqref{infgen} is a $\rho$-minimizer of the perimeter for an error function $\rho$. The error function $\rho$ is defined by selecting the function equivalent to $r \mapsto C r^{c}$ with the smallest possible $c>0$ among
\begin{itemize}
	\item $r \mapsto  Cr^d$ and $v$ if \ref{sca} holds (C depending on the constants appearing in \ref{sca}),
	\item $g_1$, $g_2$ and $v$ if \ref{eps} holds.
\end{itemize}

\end{proposition}

\noindent
Allow us to comment on (S4). We rely on the classical idea that if for some $i \geq 1$, $E^i \subset \R^d$ is the component of a generalized minimizer of \eqref{infgen}, then it is a minimizer of \eqref{infclass} with the constraint $m = |E^i|$. We then need to use either \ref{sca} or \ref{eps} to relax the mass constraint in order to be able to compare $E^i$ with a set $E'$ such that $E^i \Delta E' \subset B_r(x)$ for some $x \in \R^d$ and $r$ small enough. The scaling hypothesis \ref{sca} is a well-known method (see e.g \cite[Proposition 4.6]{NovOno22} or \cite[Proposition 3.3]{CanGol22}), but when $V$ is not necessarily positive, we lose the fact that the boundedness of $\mathcal{E}$ implies boundedness of $P$ and $V$ so that additional hypotheses are needed to control the growth of $V$. The first point of Definition \ref{volfix} appearing in \ref{eps} is inspired by the classical ``volume-fixing variations" lemma (or ``Almgren's lemma'' (see \cite[Lemma 17.21]{Maggi12}). Let us also point out that the second point of Definition \ref{volfix} is to account for perimeters with nonlocal properties, as one may take $g_2= 0$ if $P$ is the classical or anisotropic perimeter (see the definitions below). 
Eventually using \ref{loc} to deal with local perturbations of $V$, we obtain that $E$ verifies \eqref{eqquasimin} for some error function $\rho$.

\subsection{Application to three perturbed isoperimetric problems }\label{sec_examples}\hfill.

Let us present some examples from the literature of perimeter and perturbation terms satisfying the sets of hypotheses (S1) to (S4), or only (S1) and (S2) in the case of the considered Dirichlet energy. In Section \ref{applications} we provide a proof of this statement for three different perturbed isoperimetric problems. Additionally, we briefly comment on the other examples mentioned below.

Regarding the perimeter, we consider its anisotropic and anisotropic nonlocal versions. For $E, U \subset \R^d$ we set
\begin{gather*}
P_{\phi}(E, U) = \int_{(\partial^* E) \cap U} \phi(\nu_E(x)) \, d\mathcal{H}^{d-1}(x), \\
P_{K}(E, U) = \int_{(E\cap U)\times E^c} K(x-y) dxdy.
\end{gather*}
The value $P_{\phi}(E)$ is well-defined if $E$ is of finite Caccioppoli perimeter, and then $\partial^* E$ denotes the reduced boundary of $E$. The anisotropy $\phi$ is a nonnegative, one-homogeneous, convex and coercive functional. In particular, there exists $0 < C_{\phi}' \leq C_{\phi}$ such that for $x \in \R^d$,
\[
C_{\phi}' |x| \leq \phi(x) \leq C_{\phi}|x|.
\]
If $\phi = |\cdot|$ we recover the classical perimeter. Regarding the nonlocal perimeter, we require that there exist $C'_K, C_K>0$ and $s \in (0,1)$ such that for $x \in \R^d$,
\[
C'_K \lt|x\rt|^{-(d+s)} \leq K(x) \leq C_K\lt|x\rt|^{-(d+s)}.
\]
We additionally require that $K \in W^{1,1}_{\mathrm{loc}}(\R^d \setminus \{0\})$ and that for $x \in \R^d$,
\[
|\nabla K(x)| \leq \lt|x\rt|^{-(d+s+1)}.
\]
Let us point out that if $U \not = \R^d$, the definition of the relative nonlocal perimeter differs from the one found in the literature (see e.g. \cite[Section 2]{CesaNov17}). When $K=|\cdot|^{-d-s}$, we recover the fractional perimeter and simply write $P_{K} = P_s$.\medskip

The perturbation terms encompassed by our hypotheses can be split into several categories.

\noindent
The first and perhaps most studied in the literature is the Riesz-type kernel: given $E, U \subset \R^d$, we consider
\[
V_{G}(E, U) = \int_{(E\cap U) \times (E \cap U)} G(x-y) \, dxdy
\]
where $G: \R^d \to \R_+$ is continuous on $\mathbb{S}^{d-1}$ and such that there exists $\beta \in (0, d)$ such that for any $t \geq 0$ and $x \in \R^d$
\[
G(tx) \leq t^{-\beta} G(x).
\] 
We refer to \cite{KnMuII13, KnMuNov16, BonaCristo14} for seminal examples where $P = \mathrm{Per}$ and $G(x) = |x|^{-\beta}$ and to \cite{FigFus15} for an example where $P =P_s$ and $G$ is explicit as well. See also \cite{FigMag11} for a study of the anisotropic case and \cite{NovOno22} for a recent development in the case $P = \mathrm{Per}$ and with general kernels.\medskip

\begin{remark}
	Proceeding as in \cite{NovOno22}, one can obtain that generalized minimizers to $P + V_{G}$ exist and have density estimates for $P = \mathrm{Per}$, $P_{\phi}$ or $P_s$ and with $G$, nonnegative, symmetric with respect to the origin, vanishing at infinity and such that 
	\[
	G(tx) \leq t G(x) \quad \text{for } x \in \R^d \text{ and } t \geq 1.
	\]
	This case is not encompassed in our setting, because with these assumptions, neither the scaling hypothesis \ref{sca} nor the volume-fixing hypothesis \ref{eps} hold. However, the mass constraint can still be dealt with using the fact that for the considered perimeters 
	\[
	P(E \cap B_R) \leq P(E) \quad \text{ for any } E \subset \R^d \text{ and } R >0.
	\]
	This classical result is a consequence of the monotonicity of the perimeter regarding intersection with convex sets, which holds for the classical, anisotropic and fractional perimeter, but not for the generalized nonlocal perimeter $P_K$.
\end{remark}

A second family of perturbation terms appears in the prescribed curvature problem. We consider
\[
V_T(E, U) = -\int_{E \cap U} T(x) \, dx,
\]
and assume that $T$ is $L$-periodic and Lipschitz continuous. Refer for instance to \cite{GolNov12} for the case $P = \mathrm{Per}$ and to \cite{CesaNov17} for the case $P = P_s$.

\begin{remark}
If one only wants to establish that (S1) and (S2) are verified for $P + V_T$, weaker hypotheses on $T$ can be considered. We use the Lipschitz continuity assumption to establish that the volume-fixing hypothesis \ref{eps} holds (see Section 4.2).
\end{remark}

Perturbation terms involving optimal transport are studied in \cite{CanGol22}. Given $p \in [1, \infty)$ and denoting by $W_p(E, F)$ the $p$-Wasserstein distance between $E, F \subset \R^d$, one can set for $U \subset \R^d$
\[
V_{\mathcal{W}}(E, U) = \inf_{|F \cap E \cap U| = 0} W_p(E \cap U, F)^p.
\]

Eventually, we can consider a Dirichlet energy as in \cite{DePhilLambPierreVel18} and show that it satisfies (S1) and (S2). Given $E \subset \R^d$, we define the Sobolev-like space 
\[
\hat H^1_0(E) = \big\{u \in H^1(\R^d) : u= 0 \ \text{ a.e. on } E^c \big\}
\]
which is a Hilbert space as it is closed in $H^1(\R^d)$. For $p \in (d, \infty)$ and  $h \in L^p(\R^d)$, the Dirichlet (or torsion) energy of $E$ is then
\[
V_{\mathrm{Dir}}(E) = \min_u \lt\{\frac{1}{2} \int_{\R^d} |\nabla u|^2 \, dx - \int_{\R^d}  uh \,dx : u \in \hat H^1_0(E) \rt\}.
\]
and given $U \subset \R^d$, we set $V_{\mathrm{Dir}}(E, U) = V_{\mathrm{Dir}}(E \cap U)$.

\begin{remark}
	It is actually possible to show that minimizers of $P_{\phi} + V_{\mathrm{Dir}}$ or of $P_K+ V_{\mathrm{Dir}}$ admit interior and exterior density estimates. However, because we focus on the set $\hat H^1_0(E)$ instead of the Sobolev space $H^1_0(E)$ (we have to introduce $\hat H^1_0(E)$ because if $E$ is not open, there may exist a set $E'$ such that $|E \Delta E'| = 0$ but $H^1_0(E') \not = H^1_0(E)$), we were not able to prove (S3) and (S4) exactly as they stand. One can proceed as in \cite[Theorem 1.1]{DePhilLambPierreVel18} and first prove that $E$ admits exterior density estimates, so that it can be correctly be identified with then open set $E^{(1)}$, and then establish interior density estimates for $E^{(1)}$.
\end{remark}

\subsection{Notation and organisation of the article}\

All constants appearing in the article depend on the dimension $d$ and on the functions $(f_i)_i, (g_j)_j, h, \eta, \rho, v$ and parameters $r, \eps$ used in the hypotheses, where $i=1,2,3$ and $j=1,2$. We denote  them with the same letter $C$ when differentiating the constants from one another is not relevant. We write $C = C(E, m)$ to specify an additional dependency on a set $E$ or a parameter $m$. In some statements we write $A \ll B$ to indicate that there exists a constant $\eps> 0$ such that if $A  \leq \eps B$ then the conclusion of the statement holds. \medskip

In Section \ref{exisgen}, we prove that the infima of \eqref{infclass} and \eqref{infgen} coincide, and that \eqref{infgen} admits solutions. In Section \ref{densest}, we first establish that $\rho$-minimizers of \eqref{infgen} have interior and exterior density estimates. We then discuss two cases where generalized minimizers of \eqref{infgen} are $\rho$-minimizers of the perimeter as well. In Section 4, we study three examples of perturbed isoperimetric problems.

\subsection*{Acknowledgments} The author wishes to express his gratitude to his PhD advisors: M. Goldman for suggesting the problem and for several stimulating conversations and B. Merlet for his many helpful comments and suggestions during the writing of this paper.

\section{Existence of generalized minimizers}\label{exisgen}

We start off by establishing Proposition \ref{infclassinfgen}, i.e. that \eqref{infclass} and \eqref{infgen} coincide. \smallskip

	\begin{proof}[Proof of Proposition~\ref{infclassinfgen}]~\
	
	Given a set $E$ with $|E|=m$, we  define the generalized set $(E^i)_i = (E, \, \emptyset, \, \dots, \, \emptyset)$ and have $\mathcal{E}_{\mathrm{gen}}((E^i)_i) = \mathcal{E}(E)$. Hence $e_{\mathrm{gen}}(m) \leq e(m)$.
	\smallskip
		
	Conversely if we let $\eps >0$, there exists $(E^i)_{i}$ admissible for \eqref{infgen} such that 
	\[
	\mathcal{E}_{\mathrm{gen}}((E^i)_i) \leq e_{\mathrm{gen}}(m)+\eps.
	\] 
	Let us show that there exists a set $E$ admissible for \eqref{infclass} such that
	\[
	\mathcal{E}(E) \leq e_{\mathrm{gen}}(m) + 5\eps.
	\]
	By \ref{vanloc} small balls have vanishing energy: there exists $\delta = \delta(\eps)$ such that if $B_{r}$ is a centered ball of radius $r>0$, then 
	\begin{equation}\label{balldelta} 
	|B_r|\le2\delta\ \implies \ \mathcal{E}(B_r)\le\eps.
	\end{equation} 
	As $(E^i)_i$ is of finite energy and mass, there exists an integer $I = I(\eps, \delta)$ large enough that
	\[
	\sum_{i=1}^I \mathcal{E}(E^i) \leq \sum_{i\geq 1} \mathcal{E}(E^i) + \eps \leq e_{\mathrm{gen}}(m) + 2\eps \quad \text{ and } \quad \sum_{i \geq I+1}|E^i| \leq  \delta.
	\]
	Combining this with the convergence at infinity \ref{vaninf} of $\mathcal{E}$, there exists $R = R(\eps, \delta, I)$ large enough that
	\begin{equation}\label{cutoff}
	\sum_{i=1}^{I} \mathcal{E}(E^i \cap B_R) \leq \sum_{i=1}^I \mathcal{E}(E^i) + \eps \leq e_{\mathrm{gen}}(m) + 3 \eps \quad \text{ and } \quad \sum_{i=1}^{I} |E^i \cap B_{R}^c| \leq \delta.
	\end{equation}
	Let $B_r$ be the centered ball with volume 
	\[
	|B_r| = \sum_{i=1}^{I} |E^i \cap B_{R}^c| + \sum_{i \geq I+1}|E^i| \le2 \delta.
	\]
	Given $L >0$, we define the set 
	\[
	E_L = \bigg[\bigcup_{i=1}^{I} \big((E^i \cap B_{R}) + i  L e_1 \big)\bigg] \bigcup \bigg[ B_r +(I+1) L e_1 \bigg].
	\]
	By construction, for $L$ large enough $|E_L| = m$. Using recursively \ref{decker} on the vanishing range of action of $\mathcal{E}$ then yields for that for $L$ large enough
	\begin{equation*}
	\mathcal{E}(E_L) \le \sum_{i=1}^I\mathcal{E}(E^i \cap B_R) + \mathcal{E}(B_r) + \eps \leq e_{\mathrm{gen}}(m) +5 \eps
	\end{equation*}
	where we used \eqref{balldelta} and \eqref{cutoff} in the last inequality. Thus 
	\[
	e(m) \leq e_{\mathrm{gen}} + 5 \eps,
	\]
	and as $\eps>0$ is arbitrary the proof is complete.
\end{proof}

\begin{remark}\label{scaling}
	The vanishing energy of small balls \ref{vanloc} does not hold for perturbation terms $V$ which are $\alpha$-homogeneous with $\alpha < 0$. However, this hypothesis can be replaced by the assumption that $P$ and $V$ are homogeneous for some reals $\alpha, \beta$ and that $V \geq 0$. Proceeding as in the proof of \cite[Proposition 3.3]{CanGol22}, one can then show that there exists $\Lambda = \Lambda (m) \geq 0$ such that 
	\[
	e(m) = \inf_E \bigg\{ \mathcal{E}(E) + \Lambda \Big| |E|-m\Big| \bigg\},
	\]
	and
	\[
	e_{\mathrm{gen}}(m) = \inf_{(E^i)_i} \lt \{\mathcal{E}_{\mathrm{gen}}((E)^i_i) + \Lambda \Big|\sum_i |E^i| -m \Big| \rt\}.
	\]
	We can subsequently reproduce the proof of Proposition \ref{infclassinfgen} without introducing a small ball to compensate mass deficit.
\end{remark}

Under the set of hypotheses $(S2)$, we can prove Theorem \ref{infgenmingen}, i.e. that generalized minimizers of \eqref{infgen} exist. Recall that given $U\subset \R^d$, the localized versions of $P$ and $V$ are denoted by $P(\cdot, U)$ and $V(\cdot, U)$. \smallskip

\begin{proof}[Proof of Theorem~\ref{infgenmingen}]~\
	
	We follow the direct method in the Calculus of Variations. First, we use a classical minimizing sequence to build a generalized set, and then establish lower semi-continuity results to prove that this generalized set is a generalized minimizer.\smallskip
	
	\textit{Step 1. Construction of a generalized set.}\smallskip
	
	Let $(E_n)_{n}$ be a minimizing sequence for \eqref{infclass}. As assumed in the statement of Theorem \ref{infgenmingen}, $(E_n)_{n}$ is also a minimizing sequence for \eqref{infgen}. Let $r_0,r_1>0$ be as in assumptions \ref{reliso} and \ref{trainvloc} on the relative isoperimetric inequality and the periodicity of $P$ and $V$. Notice that $B_{r_0}$ contains the centered cube of side-length $r_1$. We consider a partition $(Q^i_n)_{i,n}$ of $\R^d$ into cubes of side-length $r_1$ and we set
	\[
	m^i_n = |E_n \cap Q^i_n| \quad \text{and} \quad M^i_n = |E_n \cap B^i_n|,
	\]
	where $B^i_n$ is the ball of radius $r_0$ with the same center as $Q^i_n$. Rearranging the sequence we assume that for every $n\geq 0$, $i \mapsto M^i_n$ is nonincreasing. \smallskip
	
	\textit{Step 1.1 }
	Let us now show that the series $\sum_i M^i_n$ is uniformly summable with respect to $n\geq 0$. Notice that there exists $C = C(r_1/r_0)$ such that for every $n\geq 0$
	\[
	\sum_{i\geq 1}\chi_{B^i_n} \leq C, \quad \text{so that} \quad \sum_{i \geq 1} M^i_n \leq C m.
	\]
	Thus as $M^i_n$ is nonincreasing in $i$, for every $I \geq 1$ and $i \geq I$ we have
	\begin{equation}\label{MandI}
	M^i_n \leq M^I_n \leq Cm/I.
	\end{equation}
	Let $\eps >0$. Recall that the function $f_1$ involved in \ref{reliso} is such that there exists $\delta = \delta(\eps)$ such that $m \leq \eps f_1(m)$ for any $m \leq \delta$. By \eqref{MandI} there exists $I = I(\delta)$ such that for any $i \geq I$ we have $M^i_n \leq \delta$. Up to reducing $\delta$ we assume without loss of generality that $|M^i_n| \leq |B_{r_0}|/2$. Then, by \ref{reliso} we have 
	\begin{equation}\label{restofM}
	\sum_{i\geq I} M^i_n \leq \eps \sum_{i\geq I} f_1(M^i_n) \leq \eps \sum_{i \geq I} P(E_n, B^i_n).
	\end{equation}
Given $n\geq 0$, we split the covering $(B^i_n)_i$ of $\R^d$ into $N$ families $\mathcal{B}^1_n,\dots, \mathcal{B}^N_n$ such that $|B^{i}_n \cap B^{j}_n| = 0$ if $B^{i}_n, B^{j}_n \in \mathcal B^{k}_n$ for some $1 \leq k \leq N$ and $i \neq j$. Notice that $N = N(r_1/r_0)$ is uniformly bounded in $n \in \N$. By  \ref{Psetop} we may write
	\begin{align}
	\nonumber
	\sum_{i \geq I} P(E_n, B^i_n) &= \sum_{k=1}^N\, \sum_{\substack{B^i_n\in\mathcal{B}^k_n,\,i \geq I}} P(E_n, B^i_n)\\
	\label{latticeballs}
	&\leq \sum_{k=1}^N P\Big(E_n, \bigcup_{B^i_n \in \mathcal{B}^k_n} B^i_n\Big) \leq \sum_{k=1}^N P(E_n) \leq N P(E_n).
	\end{align}
	By \ref{bdEbdP} $\sup_n \mathcal{E}(E_n) < \infty$ implies that $\sup_n P(E_n) < \infty$. Thus plugging \eqref{latticeballs} into \eqref{restofM} yields
	\[
	\sum_{i \geq I} M^i_n \leq \eps N P(E_n) \leq \eps C'
	\]
	for some constant $C'>0$. This proves that the series $\sum_i M^n_i$ is uniformly summable with respect to $n\geq 0$. As for any $i \geq 1$ and $n\geq 0$, we have $m^i_n \leq M^i_n$, the series $\sum_i m^i_n$ is uniformly summable	 in $n\geq 0$ as well.\medskip

	\textit{Step 1.2.}
	By the previous substep, there exists a sequence $(m_i)_{i\geq1}$, such that up to extraction $m^i_n \to m^i$ as $n \to \infty$ for every $i \geq 1$. Besides, $m_i\ge0$ for every $i$ and by uniform summability,
	\[
	\sum_{i} m^i = m.
	\]
	We now build a generalized set of total mass $m$. Let $x^i_n$ be the center of $Q^i_n$. Up to further extraction, we assume that for every $i,j \geq 1$, $|x^i_n - x^j_n| \to d_{i,j} \in [0, \infty]$ as $n \to \infty$. Recall that $r_1$ is chosen so that \ref{trainvloc} holds, so that $\sup_n P(E_n - x^i_n) = \sup_n P(E_n) <  \infty$. Thus by the compactness  assumption \ref{com}, for every $i\geq $1, there exists $E^i$ such that $E_n - x^i_n \to E^i$ in $L^1_{\mathrm{loc}}$. 
	
	We now define an equivalence class in the set $\{1,2,\dots\}$ by setting
	\[
	i \sim j \quad \text{ if }\quad d_{i,j} < \infty.
	\] 
	Notice that if $i \sim j$, then $E^i$ and $E^j$ coincide up to a translation. We denote by $\mathcal{C}$ the set of all equivalence classes. For every equivalence class $c \in \mathcal{C}$ let $m_c = \sum_{i \in c} m_i$ so that
	\begin{equation}\label{sumofmass}
	\sum_{c \in \mathcal{C}} m_{c} = \sum_{i \geq 1} m_i = m. \medskip
	\end{equation}
	
	\textit{Step 1.3.}
	Let us fix $c \in \mathcal C$ and let us establish that $|E^i| = m_c$ for every $i \in c$. Given $\ell \geq 1$, by definition, there exists $R_\ell$ such that for all $n\geq 0$
	\[
	\bigcup_{1\le j\le \ell,\,j\in c} Q_n^j \subset B_{R_\ell}(x_n^i).
	\]
	Thus
	\begin{align*}
	\sum_{1\le j\le \ell,\, j\in c}m_n^j = \sum_{1\le j\le \ell,\, j\in c} |E_n \cap Q_n^j | = & \, \bigg|E_n \bigcap \Big( \bigcup_{\substack{1\le j\le \ell,\,j\in c}} Q_n^j \Big) \bigg|\\
	&\leq |E_n \cap B_{R_\ell}(x_n^i)| = |(E_n - x_n^i)\cap B_{R_\ell}|.
	\end{align*}
	
	\noindent	
	As $E_n - x_n^i \to E^i$ in $L^1_{\mathrm{loc}}$, taking $n \to \infty$ yields 
	\[
	\sum_{1\le j\le \ell,\,j\in c} m^{j} \leq |E^i \cap B_{R_\ell}| \leq |E^i|.
	\]
	Letting $\ell \to \infty$ we finally obtain
	\begin{equation}\label{oneineq}
	m_c \leq |E^i|.
	\end{equation}
	
	Let us prove the converse inequality. For this, thanks to \eqref{sumofmass} and \eqref{oneineq} it is sufficient to establish the inequality  
	\begin{equation}\label{decadix}
		\sum_{c\in\mathcal{C}} |E^{i_c}| \leq m,
		\end{equation}
where for each $c$ we select one  $i_c\in c$, for instance $i_c=\min\{j : j\in c\}$. Let us fix  $N\ge1$ and define  $\mathcal{C}_N=\{c\in\mathcal{C}:i_c\le N\}$, which is a finite subset of $\mathcal{C}$. Given $R>0$, by definition of the equivalence relation for $n$ large enough $|B_R(x_n^{i_c}) \cap B_R(x_n^{i_{c'}})| =0$ for $c,c'\in \mathcal{C}_N$ with $c \neq c'$. Hence
	\[
	m  = |E_n| \geq \lt|E_n \bigcap \bigcup_{c\in\mathcal{C}_N} B_R(x_n^{i_c})\rt| = \sum_{c\in\mathcal{C}_N} |E_n \cap B_R(x^{i_c}_n)| = \sum_{c\in\mathcal{C}_N} |(E_n - x_n^{i_c}) \cap B_R|.
	\]
	
	\noindent
	Passing to the limit in $n\to \infty$ yields 
	\[
	m \geq \sum_{c\in\mathcal{C}_N} |E^{i_c} \cap B_R|.
	\]
	Eventually, letting $R\to \infty$ and then $N\to\infty$ proves \eqref{decadix}.

	Consequently for $c\in\mathcal{C}$ and $i\in c$ we have $|E^i| = m_c$. Relabeling, we write $\{E^{i_c} :c\in \mathcal C\}=\{\widetilde E^1, \widetilde E^2,\widetilde E^3,\dots\}$ so that $(\widetilde{E}^i)_i$ is admissible for \eqref{infgen}. Given $i \ge 1$, we also denote $\tilde x_n^i=x_n^j$ where $j\ge1$ is such that $E^j=\widetilde{E}^i$.
	
	\medskip
	
	\textit{Step 2 : Lower semi-continuity of the energy.}
	
	\noindent
	We are left with the proof of 
	\begin{equation*}
	\mathcal{E}_{\mathrm{gen}}((\widetilde{E}^i)_i)\le \liminf_{n\to \infty} \mathcal{E}(E_n).
	\end{equation*}
	Keeping the notation of the previous step, we let $I\geq 1$ and consider the family $\tilde{x}^1_{n}, \cdots, \tilde{x}^I_{n}$. Note that if we let $R >0$, for $n$ large enough $\min_{i \neq j} |\tilde{x}_n^i-\tilde{x}_n^j| \geq 5R$.
	\medskip
	
	We start with the perimeter term. Using the periodicity assumption \ref{trainvloc} and then the set operations property \ref{Psetop}, we have
	\begin{equation*}
	\sum_{i=1}^I P(E_n-\tilde{x}^i_n, B_R) =  \sum_{i=1}^I P(E_n, B_R(\tilde{x}^i_n)) \leq P\Big(E_n, \bigcup_{i=1}^I B_R(\tilde{x}^i_n)\Big) \leq P(E_n).
	\end{equation*}
	
	\noindent
	 Recall that $E_n-\tilde{x}^i_{n} \to \widetilde{E}^i$ in $L^1_{\mathrm{loc}}$ as $n \to \infty$. Using the lower semicontinuity  and Beppo-Levi assumptions \ref{semconloc}\&\ref{bep} in that order, letting $n\to \infty$ and then $R \to \infty$ we obtain
	 \[
	 \sum_{i=1}^I P(\widetilde{E}^i ) \leq P(E_n).
	 \]
	Thus sending $I \to \infty$ yields
	\begin{equation}\label{sciP}
	\sum_{i\geq1} P(\widetilde{E}^i) \leq \liminf_n P(E_n).
	\end{equation}
	
	Let us turn to the perturbation term. Using the functions $\eta_1, \eta_2$ of the weak superadditivity assumption \ref{wksupadd}, we write
	\begin{align*}
	\sum_{i=1}^I V(E_n - \tilde{x}^i_n, B_R) &= \sum_{i=1}^I V(E_n, B_R(\tilde{x}^i_n))  \\
	&\leq V(E_n) + \eta_1\lt(\lt|E_n \setminus \bigcup_{i=1}^I B_R(\tilde{x}^i_n)\rt|\rt)+ \eta_2 \lt(\min_{i \neq j} |\tilde{x}^i_n - \tilde{x}^j_n|-2R\rt).
	\end{align*}
	Notice that 
	\[\lt|E_n \setminus \bigcup_{i=1}^I B_R(\tilde{x}^i_n)\rt| = |E_n| - \sum_{i=1}^I |(E_n - \tilde{x}^i_n) \cap B_R|.
	\]
	Recall that $\eta_2(r) \to 0$ as $r\to \infty$ and that $\eta_1$ is continuous. Letting $n\to \infty$ in the previous inequality and using \ref{semconloc} yields
	\[
	\sum_{i=1}^I V(\widetilde{E}^i, B_R) \leq \liminf_n \sum_{i=1}^I V(E_n - \tilde{x}^i_n, B_R) \leq \liminf_n V(E_n) + \eta_1 \lt(m - \sum_{i=1}^I |\widetilde{E}^i \cap B_R|\rt).
	\]
	Notice that by \eqref{sumofmass}, letting $R \to \infty$  and then $I \to \infty$ we have
	\[
	m - \sum_{i=1}^I |\widetilde{E}^i \cap B_R| \to 0.
	\]
	Therefore, using that $\eta_1(t) \to 0$ as $t\to 0$ and letting $R\to \infty$ and then $I \to \infty$ we obtain from \ref{bep} that
	\begin{equation*}
	\sum_{i=1}^{\infty} V(\widetilde{E}^i) \leq \liminf_n V(E_n).
	\end{equation*}
	Combining this inequality with \eqref{sciP} yields
	\[
	\mathcal{E}_{\mathrm{gen}}((\widetilde{E}^i)_i) = \sum_{i=1}^{\infty}\big[ P(\widetilde{E}^i) + V(\widetilde{E}^i)\big] \leq \liminf_n  \big[V(E_n) + P(E_n)\big] = e_{\mathrm{gen}}(m).
	\]
This proves that $(\widetilde{E}^i)_i$ is a generalized minimizer of \eqref{infgenmingen}.
\end{proof}

\section{Density estimates for perturbed isoperimetric problems}\label{densest}

\subsection{Density estimates for $\rho$-minimizers of the perimeter}~\

In this subsection, we establish Theorem \ref{densityestimates}, i.e. that $\rho$-minimizers admit density estimates under the set of hypotheses (S3). \medskip

\begin{proof}[Proof of Theorem~\ref{densityestimates}]~\
	
	Take $E$ as in the statement of Theorem \ref{densityestimates}, $\eps_1 > 0$ provided by \ref{cro} and for $x \in \R^d$, $r>0$, set $m(r) = |E \cap B_r(x)|$. We will show that there exists $r_0>0$ (depending on the functions and parameters $r, \eps$ appearing in the hypotheses) such that:
	\begin{equation}\label{campanatoin}
		\text{if for some } r \leq r_0, \ \frac{m(r)}{r^d} \leq \eps_1, \text{ then }  \frac{m(r/2)}{(r/2)^d} \leq \eps_1,
	\end{equation}
	and
	\begin{equation}\label{campanatoout}
		\text{if for some } r \leq r_0, \  \frac{m(r)}{r^d} \geq \eps_1, \text{ then } \frac{m(r/2)}{(r/2)^d} \geq \eps_1.
	\end{equation}
	Combining \eqref{campanatoin} and \eqref{campanatoout} and the definitions of $E^{(1)}$ and $E^{(0)}$ then yields \eqref{intdens} and \eqref{outdens}.\medskip

	We start by proving \eqref{campanatoin}: assume that $m(r) \leq \eps_1 r^d$ for some $r>0$ to be fixed later. Translation invariance does not necessarily hold, but up to a change of coordinates we may assume that $x=0$. Notice that $t \mapsto P(B_t, E)$ can not be strictly greater than its mean value over $[r/2, r]$ for any $t \in [r/2, r]$. Hence there exists $t \in [r/2, r]$ such that
	\begin{equation}\label{PBt}
	\frac{r}{2} P(B_t, E) \leq \int_{r/2}^{r} P(B_s, E)ds \leq \int_{0}^{r} P(B_s, E)ds \leq r f_2(r, m(r)),
	\end{equation}
	where the last inequality comes from \ref{intdif}. Up to multiplying $f_2$ by a constant, we omit the factor $1/2$ in what follows. Now, applying Definition \ref{quasimin} with $F = E \setminus B_t$ yields
	\[
	P(E) \leq P(E \setminus B_t) + \rho(t).
	\]
	By applying \ref{loccom} to $P(E \setminus B_t)$ and plugging it into the previous inequality we have
	\[
	P(E, B_t) \leq C P(B_t, E) + \rho(t).
	\]
	Then, using the monotonicity \ref{Psetop} of $U \mapsto P(E, U)$ and the one of $\rho$,
	\[
	P(E, B_{r/2}) \leq C P(B_t, E) + \rho(r).
	\]
	Together with \eqref{PBt}, we obtain (again replacing $C f_2$ by $f_2$)
	\begin{equation}\label{PBr2}
	P(E, B_{r/2})  \leq \overline{f}_2(r, m(r)),
	\end{equation}
	where 
	\[
	\overline{f}_2(r) = f_2(r, m(r)) + \rho(r).
	\]
	Recall that $m(r/2) \leq m(r) \leq  \eps_1 r^d$. Without loss of generality, we may assume that $\eps_1 \leq \omega_d/2^{d+1}$, which implies $m(r/2) \leq |B_{r/2}|/2$, so that in particular $m(r/2) \leq |B_{r/2} \setminus E|$. Combining the relative isoperimetric inequality \ref{reliso} and \eqref{PBr2} yields
	\[
	f_1(m(r/2)) \leq P(E, B_{r/2}) \leq \overline{f}_2(r, m(r))
	\]
	so that
	\[
	m(r/2) \leq \overline{f}_2(r, m(r)) \frac{m(r/2)}{f_1(m(r/2))} \leq \frac{\overline{f}_2(r, m(r))}{f_1(m(r))} \, m(r),
	\]
	where we used the fact that $m \mapsto f_1(m)/m$ is nonincreasing in the last inequality. By hypothesis, $m(r) \leq \eps_1 r^d$, and recalling that $f_3 = 2^d(f_2  + \rho )/f_1 = 2^d\overline{f}_2/f_1$ we obtain
	\begin{equation}\label{fracf5}
	\frac{m(r/2)}{(r/2)^d} \leq f_3(r, m(r)) \, \frac{m(r)}{r^d} \leq f_3(r, m(r)) \eps_1.
	\end{equation}
	By contradiction, assume that $m(r/2) > \eps_1 (r/2)^d $. Then $m(r) \geq m(r/2) > \eps_1(r/2)^d$. Thus by \ref{cro}, we have $f_3(r, m(r)) \leq 1$ and by \eqref{fracf5}, $m(r/2) \leq \eps_1 (r/2)^d$, which is absurd. Hence $m(r/2) \leq \eps_1 (r/2)^d$, proving \eqref{campanatoin}.
	
	\medskip
	
	To establish \eqref{campanatoout}, we define $m^c(r) = |E^c \cap B_r|$ and assume that $m^c(r) \leq \eps_1 r^d$ for some $r > 0$. Again, applying the mean value theorem to $r\mapsto P(B_r, E^c)$ yields the existence of $t \in [r/2, r]$ such that
	\[
	P(B_t, E^c) \leq f_2(r, m^c(r)).
	\]
	Next, comparing the $\rho$-minimizer $E$ with $F = E \cup B_t$ yields
	\begin{equation*}
		P(E) \leq P(E \cup B_t) + \rho(t).
	\end{equation*}
	Using \ref{loccom} to bound the local variations of the perimeter, we obtain
	\begin{equation*}
		C P(E^c, B_t ) \leq P(B_t, E^c) + \rho(t).
	\end{equation*}
	The proof of \eqref{campanatoout} is then exactly as the one of \eqref{campanatoin} with $E$ replaced by $E^c$.
\end{proof}

\smallskip

\subsection{Generalized minimizers as $\rho$-minimizers of the perimeter}~\

In this subsection, we prove Proposition \ref{mingentoquasimin}, which describes two cases where generalized minimizers of \eqref{infgen} are also $\rho$-minimizers of the perimeter for some function $\rho$.\medskip

\begin{proof}[Proof of Proposition~\ref{mingentoquasimin}]~\
	
We first observe that given $i \geq 1$, if $E^i$ is a component of a generalized minimizer of mass $|E^i| = m_i$, then it is a minimizer of \eqref{infclass} with the mass constraint $m=m_i$. We now show that $E^i$ is a $\rho$-minimizer of the perimeter and split the proof on whether hypothesis \ref{sca} or \ref{eps} holds. To ease the notation, we write $E = E^i$ and $ m = m_i$.\medskip

\textit{Case 1 : Scaling.}

\noindent
We first assume that $P$ and $V$ admit the scaling property given by \ref{sca}. Let us establish that for some $\Lambda \gg 1$, $E$ is a minimizer of
\begin{equation}\label{defElambda}
	 \inf_{E'} \lt \{\mathcal{E}_{\Lambda} (E') = \mathcal{E}(E') + \Lambda\big| m - |E'|\big| \rt\}.
\end{equation}
By contradiction, let us assume that there exists $\Lambda_n \to \infty$ and $(E_n)_{n \in \N}$ such that 
\begin{equation}\label{bycontradic}
\mathcal{E}_{\Lambda_n}(E_n) < \mathcal{E}(E).
\end{equation}
We first notice that we must have $|E_n| \not = m$. \medskip

\textbf{Step 1: Boundedness of $P$ and $V$.} Let us show that
\begin{equation}\label{boundedinEn}
	\sup_n P(E_n) < \infty \quad \text{ and } \quad \sup_n |V(E_n)| < \infty.
\end{equation}
By \eqref{addhypquantV} from \ref{sca}, for every $n \in \N$ we have
\begin{equation}\label{toapplyYoung}
V(E_n) \geq -C_1 |E_n|^{\delta}P(E_n)^{1-\delta}.
\end{equation}

\smallskip
If $\delta = 0$, then $C_1 < 1$ and by \eqref{bycontradic} 
\begin{equation*}
	P(E_n)(1-C_1) \leq P(E_n) + V(E_n) = \mathcal{E}(E_n) \leq \mathcal{E}_{\Lambda_n}(E_n) < \mathcal{E}(E) \leq \mathcal{E}(B_{\ell(m)}),
\end{equation*}
where $B_{\ell(m)}$ is the centred ball of volume $m$. Hence \eqref{boundedinEn} holds. \medskip

If $\delta \in (0,1]$, applying Young's inequality to \eqref{toapplyYoung} yields that for every $n \in \N$
\begin{equation}\label{consYoungineq}
V(E_n) \geq -C_1( \delta |E_n| + (1-\delta)P(E_n)), 
\end{equation}
so that
\begin{equation*}
(1- C_1(1-\delta)) P(E_n) \leq  \mathcal{E}(E_n) + \delta C_1|E_n|.
\end{equation*}
Thus for $\Lambda_n \geq \delta C_1$, by the triangle inequality
\begin{equation*}
	(1- C_1(1-\delta))P(E_n) \leq  \mathcal{E}(E_n) +\Lambda_n |m - |E_n|| + \delta C_1 m \leq  \mathcal{E}_{\Lambda_n}(E_n) + \delta C_1 m.
\end{equation*}
Therefore by \eqref{bycontradic} $\sup_n P(E_n) < \infty$. By \eqref{toapplyYoung}, up to relabelling $C_1$ we have 
\begin{equation}\label{estimV}
V(E_n) \geq -C_1 |E_n|^{\delta}.
\end{equation} 
Therefore to prove \eqref{boundedinEn} it is sufficient to establish that $\sup_n|E_n| < \infty$. By \eqref{consYoungineq} and replacing $C_1$ by $\max(1, C_1)$
\[
|E_n|(\Lambda_n - \delta C_1) \leq V(E_n) + C_1(1-\delta)P(E_n) + \Lambda_n |E_n| \leq C_1 \mathcal{E}_{\Lambda_n}(E_n) + C_1 \Lambda_n m,
\]
so that dividing by $\Lambda_n$ yields
\[
|E_n|(1 - \delta C_1 \Lambda_n^{-1}) \leq C_1 \Lambda_n^{-1} \mathcal{E}_{\Lambda_n}(E_n) + C_1 m.
\]
Hence $\sup_n |E_n| < \infty$ and \eqref{boundedinEn} holds.\medskip

\textbf{Step 2: Showing that $E$ minimizes $\mathcal{E}_{\Lambda}$.} We are now ready to prove \eqref{defElambda}. We set $t_n^d = m |E_n|^{-1}$ so that $|t_n E_n| = m$ and write $t_n = 1 + \eps_n$  where $\eps_n\in (-1, +\infty)$. Combining \eqref{bycontradic} and \eqref{boundedinEn} one has
\begin{equation*}
	\Lambda_n ||E_n|-m | \leq \mathcal{E}(E) - P(E_n) - V(E_n) \leq \mathcal{E}(B_{\ell{m}}) - V(E_n),
\end{equation*}
so that
\begin{equation*}
\sup_n 	\Lambda_n ||E_n|-m | \leq \sup_n  \lt[\mathcal{E}(B_{\ell{m}}) - V(E_n)\rt] = \mathcal{E}(B_{\ell(m)}) - \inf_n V(E_n) < \infty
\end{equation*}
As $\Lambda_n \to \infty$ as $n \to \infty$, the previous inequality implies that $|E_n| \to m $ and $\eps_n \to 0$ as $n \to \infty$. Now using the scaling part of \ref{sca}, by definition of $E$ we have
\begin{equation*}
	\mathcal{E}_{\Lambda_n} (E_n) < \mathcal{E}(E) \leq \mathcal{E}(t_n E_n) \leq t_n^{\alpha}P(E_n) + t_n^{\beta}V(E_n) = (1+\eps_n)^{\alpha}P(E_n) + (1+\eps_n)^{\beta}V(E_n).
\end{equation*}
Therefore, a Taylor expansion yields  that for some $C_3 = C_3(\alpha, \beta, \delta, \gamma, m)>0$
\begin{equation}\label{lambdam}
\Lambda_n \big|m - |E_n|\big| \leq |\eps_n|  \mathcal{E}(E_n) \leq |\eps_n| C_3 \mathcal{E}(B_{\ell(m)}) = |\eps_n| C_3.
\end{equation}
Finally, notice that by definition of $\eps_n$,
\[
\big|m - |E_n|\big| = m|1 - t_n^{-d}|= m |1-(1+ \eps_n)^{-d}| = m \eps_n +  O(\eps_n^2).
\]
Injecting this last equation into \eqref{lambdam}, we obtain
\[
\Lambda_n m |\eps_n| \leq |\eps_n| C_3
\] 
so that $\Lambda_n\leq C_3$, contradicting the fact that $\Lambda_n \to n$ as $n \to \infty$. We thus have that $E$ minimizes \eqref{defElambda} for $\Lambda \gg 1$.\medskip

\textbf{Step 3 : Conclusion.} We finally consider $E'$ with $E' \Delta E \subset B_r(x)$ for some $x \in \R^d$ and $r>0$. Let us consider $\Lambda \gg 1$ such that $E$ minimizes \eqref{defElambda}. We have 
\begin{equation*}
P(E) \leq P(E') + \lt[V(E') - V(E) \rt]+ \Lambda \big||E'| - m  \big| \leq P(E') + \lt[V(E') - V(E) \rt] + \Lambda  \omega_d r^d.
\end{equation*}
Using the local control \ref{loc} on $V$, for $r \ll1$ we have
\[
P(E) \leq P(E') + v(r) + \Lambda \omega_d r^d.
\]
Therefore $E$ is a quasi-minimizer of the perimeter. Its error function $\rho$ is $v$ or $r \mapsto \Lambda \omega_d r^d$, whichever is equivalent to $r \mapsto C r^{c}$ with the smallest possible $c>0$.\medskip

\textit{Case 2 : Local variation.}

\noindent
We now assume that \ref{eps} holds. Given $x \in \mathbb{R}^d$ and $0 < r \leq r_5/2$, we consider $E'$ such that $E' \Delta E \subset B_r(x)$. We notice that $\big| |E'|-|E|\big| \leq \omega_d r_5^d$ so that for $r_5$ small enough we may write $|E'| = |E| + \eps$ with $|\eps| \leq \eps_2$. By Definition \ref{volfix} (1), there exist  $x_0 \in \R^d$ and $F \subset \R^d$ such that $|F| = |E'|-2 \eps = |E| - \eps'$ and $E \Delta F \subset B_{r_5}(x_0)$ with $|B_{r_5}(x) \cap B_{r_5}(x_0)| = 0$ and
\begin{equation}\label{firstpartvolfix}
	\mathcal{E}(F) \leq \mathcal{E}(E) + g_1(\omega_d r^d). 
\end{equation}
We then define
\begin{equation*}
	F' = \big(F \cap B_{r}(x_0)\big) \cup \big(E \setminus (B_r(x) \cup B_{r}(x_0))\big)\cup \big(E' \cap B_r(x) \big)
\end{equation*}
and observe that $|F'| = |E|$, $F \Delta F' = E \Delta  E' \subset B_{r}(x)$ and $E \Delta F = E' \Delta F' \subset B_{r}(x_0) $. By Definition \ref{volfix} (2),
\begin{equation*}
	P(F') - P(F) \leq P(E') - P(E) +  g_2(r),
\end{equation*}
so that by minimality of $E$ and the locality assumption of $V$ \ref{loc}:
\[
\mathcal{E}(E) - \mathcal{E}(F) \leq \mathcal{E}(F') - \mathcal{E}(F) \leq P(E') - P(E) + g_2(r) + v_3(r).
\]
Injecting this inequality into \eqref{firstpartvolfix} yields
\begin{equation*}
	P(E) \leq P(E') + g_1(\omega_d r^d)+ g_2 (r) + v_3(r),
\end{equation*}
so that $E$ is a quasi-minimizer of the perimeter. Its error function $\rho$ is $g_1, g_2$ or $v$, whichever is equivalent to $r \mapsto C r^{c}$ with the smallest possible $c>0$.

\end{proof}

\section{Application to perturbed isoperimetric problems}\label{applications}

In this section, we consider three perturbed isoperimetric problems and investigate whether they satisfy the sets of hypotheses (S1) to (S4). We let the reader check that the hypotheses not mentioned in the proofs are indeed verified.

\subsection{An anisotropic liquid drop model}\	

\medskip
\noindent
Let $\phi$ and $G$ be as in Section \ref{sec_examples}. We consider for $E \subset \R^d$ and $U$ open,
\begin{equation*}
P_{\phi}(E, U) = \int_{(\partial^* E) \cap U} \phi(\nu_E(x)) d\mathcal{H}^{d-1}(x) \quad \text{and} \quad V_{G}(E, U) = \int_{(E \cap U) \times (E \cap U)} G(x-y) dxdy.
\end{equation*}

\begin{proposition}\label{PphiVG1}
The perimeter term $P_{\phi}$ and perturbation term $V_{G}$ satisfy (S1) to (S4).
\end{proposition}

\begin{proof}

Regarding $(S2)$, the isoperimetric inequality  \ref{reliso} and compactness property \ref{com} for $P_{\phi}$ are implied by the fact that $C_{\phi}' \, \mathrm{Per} \leq P_{\phi} \leq C_{\phi} \, \mathrm{Per}$. We can thus take $f_1(m) = C_1 m^{(d-1)/d}$. As for the weak superadditivity \ref{wksupadd} of $V_{G}$, let us denote $V_{G}(E) = L_{G}(E,E)$ where $L_{G}(E,F) = \int_{E\times F} G(x-y)dxdy$ and given some $x^i \in \R^d$ we set $B_i = B_R(x^i)$. We compute
\begin{equation}\label{wksu}
L_{G}(E \cap \cup_i B_i, E \cap \cup_i B_i) = \sum_{i=1}^I L_{G}(E \cap B_i, E \cap B_i)+ \sum_{i \not = j} L_G(E \cap B_i, E \cap B_j).
\end{equation}
Thus
\[
V_{G}(E) \geq V_{G}(E, \cup_i B_i) = \sum_{i=1}^I V_{G}(E, B_i) + \sum_{i \neq j} L_G(E \cap B_i, E \cap B_j)
\]
and we conclude that \ref{wksupadd} holds with $\eta_1 = \eta_2 = 0$. \medskip

We now prove (S4) and finish with (S3). The scaling hypothesis \ref{sca} is verified by $P_{\phi}$ and $V_{G}$ by hypothesis on $\phi$ and $G$. Regarding \ref{loc}, first notice that by hypothesis on $G$
\begin{equation}\label{hyponG1}
	G(x) \leq  \frac{C}{|x|^{\beta}} \quad \text{for any } x \in \R^d, \text{ where } C = \sup \lt \{G(x) : x \in \mathbb{S}^{d-1} \rt\}.
\end{equation}
Now given $E, E' \subset \R^d$ such that $E \Delta E' \subset B_r(x)$ for some $x\in \R^d$ and $r>0$, we compute
\begin{align*}
	V_{G}(E) - V_{G}(E') &= L_{G}(E, E) - L_{G}(E',E')\\
	&= L_{G}(E \setminus E', E) + L_{G}(E\cap E', E) - L_{G}(E'\setminus E, E') - L_{G}(E' \cap E, E')\\
	&\leq L_G(E \setminus E', E) + L_{G}(E \cap E',  E \setminus E').
\end{align*}
Thus by symmetry of the roles of $E$ and $E'$, and defining $E_{\Delta} = E \Delta E'$ and $E_{\cup} = E \cup E'$,
\begin{equation*}
	\lt| V_{G}(E)-V_{G}(E') \rt| \leq L_{G}(E \cap E', E \Delta E') + L_{G}(E \Delta E', E \cup E')
	\leq C \int_{E_{\Delta} \times E_{\cup}} \frac{dxdy}{|x-y|^{\beta}},
\end{equation*}
where the last inequality is a consequence of \eqref{hyponG1}. By \eqref{hyponG1},  there also exists $R = R(G) >0$ such that $G\leq 1$ outside of $B_R$. As $\beta\in(0,d)$,  \eqref{hyponG1} implies that $G$ is integrable on $B_R$. We define for $R >0$the set $ S_R = \{(x,y) \in \R^d \times \R^d, \, |x-y| < R\}$ and we have
\begin{align*}
\int_{E_{\Delta} \times E_{\cup}} \frac{dxdy}{|x-y|^{\beta}} &\leq \int_{(E_{\Delta} \times E_{\cup}) \cap S_R} G(x-y) dydx + \int_{(E_{\Delta} \times E_{\cup}) \cap S_R^c} G(x-y) dydx\\
&\leq \int_{E \Delta E'} \int_{B_R} G(z) dz + |E \Delta E'| |E \cup E'| = C(G, m) |E \Delta E'|,
\end{align*}
so that \ref{loc} holds with $v(r) = Cr^d$. Thus by Proposition \ref{mingentoquasimin}, generalized minimizers of $P_{\phi} + V_{G}$ are quasi-minimizers of the perimeter with error function $\rho(r) = C r^d$.

As for $(S3)$, the fact that $P_{\phi}$ satisfies \ref{loccom} on local comparisons is classical and a consequence of \cite[Theorem 16.3]{Maggi12}. Thanks to the fact that $P_{\phi} \leq C_{\phi} \mathrm{Per}$,  the integral inequality \ref{intdif} is verified with $f_2(r, m) = C_3 m/r$. Finally, regarding the density scale factor assumption \ref{cro}, recall that
\[
f_1(m) = C_1 m^{\frac{d-1}{d}}.
\] 
Therefore, up to replacing $f_3$ by $Cf_3$ we have
\[
f_3(r, m) = \frac{f_2(r, m) + \rho(r)}{f_1(m)} =  \frac{m^{\frac{1}{d}}}{r} + r \Big(\frac{m}{r^d} \Big)^{\frac{1-d}{d}}= \Big(\frac{m}{r^d}\Big)^{\frac{1}{d}}  + r \Big(\frac{m}{r^d} \Big)^{\frac{1-d}{d}} .
\]
Taking for instance $r_3 = 2^{-d}$ and $\eps_1 = 2^{-d}$, we obtain that if $r \leq r_3$ and $m$ is such that $\eps_1 2^{-d} \leq m r^{-d} \leq \eps_1$, then $f_3(r, m) \leq 1$. Therefore, \ref{cro} holds.
\end{proof}

\begin{remark}
	Another possible nonlocal kernel is of the form
	\[
	V_{K}(E, U) = -P_K(E, U) =- \int_{E \cap U \times E^c} K(x-y) \, dxdy,
	\]
	so that the corresponding isoperimetric problem $P - P_K$ may be seen as the difference between a (local or nonlocal) perimeter and a nonlocal perimeter. Notice that if $K \in L^1(\R^d)$ we may write
	\[
	V_{K}(E, U) = \int_{(E \cap U) \times E} K(x-y) \, dxdy - |E \cap U| \|K \|_{L^1(\R^d)},
	\]
	so that the analysis of this case is exactly as in the case $V = V_{G}$ studied above. However in the recent works \cite{DiCNovRufVal, GolMerPeg22, MelWu22} the considered problem is
	\[
	\omega_{d-1} \textrm{Per}(\cdot) - (1-s) P_s(\cdot) \quad \text{or} \quad (1-t)P_t(\cdot) - (1-s)P_s(\cdot)
	\]	
	with $0 < s < t <1$, so that the assumption $V_{K} \in L^1(\R^d)$ cannot be used. While various sets of hypotheses on $K$ are used in the aforementioned articles, proving that (S1) and (S2) hold is similar to the case $V = V_{G}$. To obtain density estimates however, the known approaches revolve around showing that there exists $0 < s_0<1$ such that for $E \Delta E' \subset B_r(x)$,
	\[
	\big| V_{K}(E) - V_{K}(E') \big| \leq C(K) |E \Delta E'|^{1-s_0} \mathrm{Per}(E \Delta  E')^{s_0}.
	\]
	The dependency on the perimeter of $E\Delta E'$ appearing in the bound on the local variations of $V$ then prevents one from establishing \ref{loc}. Thus (S4) does not hold even though assumption \ref{sca} is verified. This problem in turn prevents us from establishing \ref{cro}, so that (S3) does not hold even though \ref{loccom} and \ref{intdif} of (S3) are verified. Usually, one has to first show density estimates for the perimeter before establishing volume density estimates, which is outside the framework of this paper.
\end{remark}

\subsection{A prescribed nonlocal curvature problem}\label{genpres}\

\medskip

\noindent
Let $K$ and $T$ be as in Section \ref{sec_examples} and fix $s\in(0,1)$. In particular, recall that $T$ is $L$-periodic for some $L >0$ and Lipschitz continuous. Given $E, U \subset \R^d$ we consider
\begin{equation*}
	P_K(E) = \int_{(E\cap U) \times E^c} K(x-y) dxdy \quad \text{and} \quad V_T(E, U) = -\int_{E \cap U} T(x) \, dx.
\end{equation*}

\begin{proposition}\label{PKVT}
	The perimeter term $P_{K}$ and perturbation term $V_T$ satisfy (S1) to (S4).
\end{proposition}

\begin{proof}
Regarding (S2), recall that by hypothesis $C'_K P_s \leq P_K \leq C_K P_s$. It is thus enough to prove that the relative isoperimetric inequality holds for $P_s$. We first observe that that for any $E \subset \R^d$, $r>0$ and $x\in \R^d$, writing by abuse of notation $B_r = B_r(x)$
\begin{equation}\label{poinca}
2 P(E, B_r) \geq 2\int_{(E\cap B_r) \times (E^c \cap B_r)} \frac{dxdy}{|x-y|^{d+s}} = \int_{B_r\times B_r} \frac{|\chi_E(x) - \chi_E(y)|^2}{|x-y|^{d+s}}.
\end{equation}
Then if $|E\cap B_r| \leq |B_r|/2$, we proceed exactly as in the proof of \cite[Lemma 2.5]{DiCNovRufVal}: by a Poincaré-type inequality for fractional Sobolev spaces, we obtain that for any $r_0>0$, there exists $C = C(r_0, d,s)$ such that for any $r \leq r_0$,
\[
P(E, B_r) \geq C |E \cap B_r|^{\frac{d-s}{d}}.
\]
Conversely, if $|E^c \cap B_r| \leq |B_r|/2$, we can proceed as in the preceding case because the roles of $E$ and $E^c$ are symmetrical in \eqref{poinca}. Eventually, we have that the relative isoperimetric inequality \ref{reliso} holds for
\[
f_1(m) = C_1 m^{\frac{d-s}{d}}.
\]
The compactness property \ref{com} also holds because of the embedding theorems for fractional Sobolev spaces. The periodicity assumption \ref{trainvloc} holds for $V_T$ because $T$ is $L$-periodic and the constraint $L = r_1 \leq 2 r_0/\sqrt{d}$ is verified by setting $r_0 = \sqrt{d} L/2$. Let us prove \ref{bdEbdP} on the boundedness of the perimeter by contraposition. We have $V_T(E) \geq -\|T\|_{\infty} |E|$, so that if there exists $(E_n)_n$ with $\sup_n P_K(E_n) = \infty$ and $\sup_n |E_n| =C < \infty$, then
\[
\mathcal{E}(E_n) =  P_{K}(E_n) + V_T(E_n) \geq P_K(E_n) - \|T\|_{\infty} |E_n|.
\]
Hence $\sup_n\mathcal{E}(E_n) = \infty$ and \ref{bdEbdP} is proved. 
\medskip 

Regarding (S4), as $P_K$ and $V_T$ have no scaling property, we have to establish that \ref{eps} on volume-fixing variations holds. Regarding Definition \ref{volfix} (1), we proceed as in \cite[Lemma 3.1]{CesaNov17}. Given a minimizer $E$ of \eqref{infclass},  we have that $|E| <  \infty$ and $P_K(E) < \infty$. We consider $E' \subset \R^d$ such that $E \Delta E' \subset B_r(x)$ for $r \leq r_5$ and $x \in \R^d$. Let us show that there exists $x_0 \neq x \in \R^d$ such that
\begin{equation}\label{xzeroborder}
	|E \cap B_{r_5}(x_0)| >0, \quad |E^c \cap B_{r_5}(x_0)| >0 \quad \text{and} \quad |B_{r_5}(x) \cap B_{r_5} (x_0)| = 0.
\end{equation}
Up to reducing $r_5$, we may assume that $\omega_d r_5^d \leq |E|/4$. Thus there exists $\theta \in \mathbb{S}^{d-1}$ such that $x_{\theta} = x + 2 r_5 \theta \in E$. Additionally, $|B_{r_5}(y) \cap B_{r_5}(x)| = 0$. Finally, there also exists $r_0 \geq 2 r_5$ such that for $x_0= x + r_0 \theta$, we have $x_0 \in E$ and $|B_{r_5}(x_0) \cap E^c |\neq 0$. Thus \eqref{xzeroborder} holds.

Hence by the relative isoperimetric inequality we have
\begin{equation}\label{defperim}
\mathrm{Per}(E, B(x_0, r_5)) = \sup \lt\{\int_E \mathrm{div}S(x) \, dx : S \in C^1_c(B(x_0, r_5), \R^d), \, \|S\|_{\infty} \leq 1 \rt\} >0.
\end{equation}
Hence there exists $S \in C^1_c(B(x_0, r_5), \R^d)$ such that $M = \small \int_E \mathrm{div}S(x) \, dx >0$.

Let us now define for $t \in (-1,1)$ the maps $\Phi_t(x) = x +t S(x)$. By a changing of variables, we compute
\begin{equation}\label{changevarPk}
P_K(\Phi_t(E)) = \int_{E\times E^c} K(\Phi_t(x) - \Phi_t(y)) \lt(1 + t \mathrm{div} S(x) + t \mathrm{div}S(y) + o(t) \rt)dxdy.
\end{equation}
By hypothesis on $K$ we may write that for any $x,y \in \R^d$ and $t \in (-1,1)$
\begin{equation*}
	 K(\Phi_t(x) - \Phi_t(y)) - K (x-y) = t\int_0^1 \nabla K\lt[x-y + ut(S(x)-S(y))\rt]\cdot (S(x)-S(y))du.
\end{equation*}
Combining the regularity of $S$ with the fact that $|\nabla K |\leq |\cdot|^{-(d+s+1)}$ yields that for some $C = C(S, K)$
\begin{equation*}\label{TaylorK}
\int_{E\times E^c}|K(\Phi_t(x) - \Phi_t(y)) - K (x-y)| dxdy \leq C |t| \int_{E\times E^c}\frac{dxdy}{|x-y|^{d+s}} \leq C |t|P_K(E).
\end{equation*}
where we used the fact that $C'_K |\cdot|^{-(d+s)} \leq K(\cdot)$ in the last inequality. Reinjecting this into \eqref{changevarPk} we obtain
\begin{equation}\label{firstvarPK}
(1- C|t|)P_K(E) \leq P_K(\Phi_t(E)) \leq (1+ C |t|)P_K(E)
\end{equation}
for some $C= C(S,K)$. We proceed similarly for the perturbative term, writing
\begin{equation*}
V_T(\Phi_t(E)) = - \int_{\Phi_t(E)} T(x) dx = -\int_E T(x + S(x)) (1 + t \mathrm{div}S(x) + o(t)) dx.
\end{equation*}
Using the Lipschitz continuity of $T$ and \eqref{firstvarPK} we obtain that for some $C = C(S,K)$:
\begin{equation*}
	\mathcal{E}(\Phi_t(E)) = \mathcal{E}(E) + C t + o(|t|).
\end{equation*}
Finally, notice that
\[
|\Phi_t(E)| = \int_{E} \lt(1 + t \, \mathrm{div}\, S(x) + o(t) \rt)dx = |E| + M t + o(t).
\]
Thus for $|\eps| \ll 1$ we can find $t(\eps) = \eps/M + o(|\eps|)$ such that $F = \Phi_{t(\eps)}(E)$ satisfies $|F| = |E| + \eps$, $E \Delta F \subset B_{r_5}(x_0)$ and
\[
\mathcal{E}(F) \leq \mathcal{E}(E) + C |\eps|,
\]
so that Definition \ref{volfix} (1) holds for $g_1(|\eps|) = C |\eps|$. 

Regarding Definition \ref{volfix} (2), let us consider $E, E', F, F' \subset \R^d$ and $r_5 >0$ such that $E \Delta E'  = F \Delta F' \subset B_{r}(x)$ for some $r \leq r_5$ and $E\Delta F = E' \Delta F' \subset B_{r_5}(x_0)$ with $B_{r_5}(x)$ and $B_{r_5}(x_0)$ disjoint. Recall that we want an estimate of the form
\begin{equation}\label{recallest}
\Delta P_K = P_K(E) - P_K(E') - (P_K(F) - P_K(F')) \leq  g_2(r),
\end{equation}
for some nondecreasing function $g_2 : \R_+ \to \R_+$. Up to a change of coordinates we assume that $x=0$ and given $A, B \subset \R^d$ we define
\[
L(A,B) =  \int_{A\times B} K(x-y) dxdy.
\]
We first decompose $E$ over $B_r, B_{r_5}(x_0)$ and $H = B_r^c \cap B_{r_0}^c$ and obtain
\begin{equation*}
L(E, E^c) = L(E \cap B_r, E^c) + L(E \cap B_{r_5}(x_0), E^c) + L(E  \cap H, E^c \cap B_r) + L(E \cap H, E^c \cap B_r^c).
\end{equation*}
We write the same formula for $L(E', E'^c)$, and thus obtain for $\Delta P_{E,F} =  P_K(E) - P_K(F)$  that
\begin{equation*}
\begin{split}
\Delta P_{E,F} \,=\, & L(E \cap B_r, B_{r_5}(x_0) \cap E^c) - L(E \cap B_r, B_{r_5}(x_0) \cap F^c)\\
&+ L(E \cap B_{r_5}(x_0), E^c) - L (F \cap B_{r_5}(x_0), F^c)\\
&+ L(E\cap H, E^c \cap B_r^c) - L(E \cap H, F^c \cap B_r^c).
\end{split}
\end{equation*}
The corresponding expression for $\Delta P_{E',F'} =  P_K(E') - P_K(F')$ is
\begin{equation*}
\begin{split}
\Delta P_{E',F'} \,= \,& L(E' \cap B_r, E'^c \cap  B_{r_5}(x_0)) - L(E' \cap B_r, F'^c \cap  B_{r_5}(x_0) )\\
&+ L(E' \cap B_{r_5}(x_0), E'^c) - L (F' \cap B_{r_5}(x_0), F'^c)\\
&+ L(E' \cap H, E'^c \cap B_r^c) - L(E'\cap H, F'^c \cap  B_r^c).
\end{split}
\end{equation*}
Notice that $E \cap H = E' \cap H$, that $E^c \cap B_r^c = E'^c \cap B_r^c$ and that $F^c \cap B_r^c = F'^c \cap B_r^c$. Hence the last lines of $\Delta P_{E,F}$ and $\Delta P_{E',F'}$ will cancel each other out in the difference $\Delta P_K = \Delta P_{E,F} - \Delta P_{E',F'}$. Therefore
\begin{equation*}
\begin{split}
\Delta P_K \,=\, &L(E \cap B_r, E^c \cap B_{r_5}(x_0)) - L(E \cap B_r, F^c \cap B_{r_5}(x_0))\\
&+ L(E' \cap B_r, F'^c \cap B_{r_5}(x_0))- L(E' \cap B_r, E'^c \cap B_{r_5}(x_0))\\
&+ L(E \cap B_{r_5}(x_0), E^c \cap B_r) - L(E \cap B_{r_5}(x_0), E'^c \cap B_r)\\
&+ L(F \cap B_{r_5}(x_0), E'^c \cap B_r) - L(F \cap B_{r_5}(x_0), E^c \cap B_r).
\end{split}
\end{equation*}
Notice that the right-hand side of the previous equality is a sum of terms of the form $L_K(A, B)$, with either $A \subset B_r,$ and $B \subset B_{r_5}(x_0)$ or vice versa. In both cases, there exists $C = C(K, r_5) > 0$ such that $\inf \{|x-y| : (x,y) \in A \times B \} \geq C$. Also recalling that $K(\cdot) \leq C_K |\cdot|^{-(d+s)}$, up to relabelling $C_K$ we have
\begin{equation}\label{intint}
L_K(A, B) \leq C\iint_{B_r \times B_{r_5}(x_0)} \frac{dzdy}{|z-y|^{d+s}} \leq C |B_r||B_{r_5}(x_0)| = C r^d
\end{equation}
Therefore $\Delta P_K \leq C r^{d}$, so that \eqref{recallest} holds with $g_2(r) = C r^{d}$.

Finally, \ref{loc} on the local Lipschitz continuity of $V$ is verified with $v(r) = \|T\|_{\infty} \omega_d r^d$. Thus by Proposition \ref{mingentoquasimin}, there exists $C = C(g_1, g_2, v)$ such that for $r \ll 1$, generalized minimizers of \eqref{infgen} are $\rho$-minimizers for $\rho(r) =  C r^{d}$. \medskip
 
We conclude with $(S3)$. Regarding the integral inequality \ref{intdif}, let $E\subset \R^d$ and $x \in \R^d$. Up to a translation, we assume that $x=0$. Now given $z \in B_u$, we write
\[
\int_{B_u^c} \frac{dy}{|z-y|^s} \leq \int_{B_u^c} \frac{dy}{(|y|-|z|)^s} \leq C \int_{u-|z|}^{\infty} \frac{v^{d-1}dv}{v^{d+s}} = \frac{C}{(u-|z|)^s}.
\]

Hence
\begin{equation*}
 P_K(B_u, E) = \int_{(E \cap B_u) \times B_u^c }\frac{dzdy}{|z-y|^s} \leq C \int_{E \cap B_u} \frac{dz}{s(u-|z|)^{s}}= C\int_{0}^u \frac{\mathcal{H}^{d-1}(E \cap B_{v})}{(u-v)^s} \, dv.
\end{equation*}
Hence by Fubini-Tonelli
\[
\begin{split}
\frac{1}{r}\int_0^r P_K(B_u, E) \, du
&=\frac{C}{r}\int_0^r \int_v^r \frac{\mathcal{H}^{d-1}(E \cap B_{v})}{(u-v)^s} \, du dv \\
&\leq \frac{C r^{1-s}}{r} \int_{0}^r \mathcal{H}^{d-1}(E \cap B_{v}) \,dv
= \frac{C}{r^s} \, |E \cap B_r|,
\end{split}
\]
and one can take $f_2(r,m) = C_3 r^{-s} m$. We finish by establishing \ref{cro} on the density scale factor. Recall that 
\[
f_1(m) = C_1 m^{\frac{d-s}{d}}
\]
so that up to a multiplicative constant
\[
f_3(r, m) = \frac{f_2(r, m) + \rho(r)}{f_1(m)} = \Big(\frac{m}{r^d}\Big)^{\frac{s}{d}} + r\Big(\frac{m}{r^d} \Big)^{\frac{s-d}{d}}.
\]
We then conclude as in the proof of Proposition \ref{PphiVG1} : for some $r_3, \eps_1>0$,
\[
f_3(r, m) \leq 1 \quad \text{for every} \quad r\leq r_3 \quad \text{and} \quad \frac{\eps_1}{2^d} < \frac{m}{r^d} \leq \eps_1.
\]
\end{proof}

\begin{remark}
Notice that excepted \ref{cro} on the density scale factor and \ref{sca} on the scaling of $P$ and $V$, all the hypotheses in (S1) to (S4) can be checked separately in $P$ and $V$. Also notice that in Propositions \ref{PphiVG1} on $P_{\phi} + V_T$  and \ref{PKVT} on $P_K + V_G$ we respectively checked that $P_{\phi}$ and $P_K$ satisfied (S1) to (S4). Thus if one wants to check that (S1) to (S4) are satisfied for $P_K + V_{G}$ or $P_{\phi} + V_T$, only \ref{cro} and either \ref{sca} or \ref{eps} have to be checked.

For \ref{sca} to hold with $P_K + V_G$, one has to add a scaling hypothesis on $K$ (which is verified if $P_K$ is the fractional perimeter $P_s$). If one instead wants to show that \ref{eps} holds, hypotheses on $G$ must be added: as for $P_K$, assuming that $G \in W^{1,1}_{\mathrm{loc}}(\R^d \setminus \{0\})$ and that $G$ and $\nabla G$ are controled by sufficently integrable functions allows to conclude. Establishing that \ref{cro} holds is identical to the case $P_{\phi} + V_G$. Indeed, having $f_1(m) = C m^{(d-s)/d}$ or $f_1(m)=  C m^{(d-1)/d}$ does not change the proof of \ref{cro}, because the perturbation introduced by $V_G$ is of the form $C m^d$ and $d > \max (d-1, d-s)$.

We finish with $P_{\phi} + V_T$. On the one hand, to show that \ref{eps} holds we have to add the hypothesis that $\phi \in C^1(\mathbb{S}^{d-1})$. Indeed, it implies that $P_{\phi}$ admits a first variation (see \cite[Exercise 20.7]{Maggi12}) so that the first point of \eqref{volfix} holds. The second point of \eqref{volfix} holds with $g_2 = 0$ because of the locality of $P_{\phi}$. On the other hand, for \ref{sca} to hold one has to add a scaling hypothesis on $T$. Lastly, one can show that \ref{cro} holds by proceeding as in the previous paragraph on $P_K + V_G$.
\end{remark}

\begin{remark}
We also mentioned in the introduction that for $p \geq 1$ and $E, U \subset \R^d$ the Wasserstein functional 
\[
V_{\mathcal{W}}(E, U) = \inf_{|F \cap E \cap U| = 0} W_p(E \cap U, F)^p
\]
could be considered as a perturbative term. Let us briefly explain why $P_{\phi} + V_{\mathcal{W}}$ satisfy (S1) to (S4). The hypotheses of (S1) and (S2) are verified for $V_{\mathcal{W}}$ as a consequence of \cite[Proposition 2.2]{CanGol22}. Additionally, $V_{\mathcal{W}}(tE) = t^{\beta} V_{\mathcal{W}}(E)$ with $\beta = p+d$, so that \ref{sca} is verified. Lastly by \cite[Lemma 2.4]{CanGol22} for $E, E' \subset \R^d$:
\begin{equation}
	|V_{\mathcal{W}}(E)- V_{\mathcal{W}}(E')| \leq C(|E|^{\frac{p}{d}} + |E|^{\frac{p}{d}}) |E \Delta E'|,
\end{equation}
so that \ref{loc} holds with $v(r) = C r^d$. We establish that \ref{cro} holds as in the previous remark.

Lastly, if one wants to study $P_K + V_{\mathcal{W}}$ in the case where $K$ admits no scaling, one will need to show that $V_{\mathcal{W}}$ admits volume-fixing variations. This last fact remains an open question for now, although we believe it may hold without additionnal assumption on $\mathcal{W}$.
\end{remark}

\subsection{A model of a nonlocal perimeter interacting with Dirichlet eigenvalues}\

\medskip
\noindent
In this example, the perimeter is $P_{\phi}$ and the perturbation term is $V_{\mathrm{Dir}}$, i.e. for $E, U \subset \R^d$
\[
V_{\mathrm{Dir}}(E, U) = V_{\mathrm{Dir}}(E \cap U)  = \min_u \lt\{\frac{1}{2} \int_{\R^d} |\nabla u|^2 \, dx - \int_{\R^d}  uh \,dx : u \in \hat H^1_0(E \cap U) \rt\},
\]
where $h\in L^p(\R^d)$ for some $p>d$. Before proving that $V_{\mathrm{Dir}}$ satisfies (S1) and (S2), let us recall some of its properties. Given $E$ with $|E|< \infty$, $V_{\mathrm{Dir}(E)}$ admits an unique minimizer $w_{E}$, which is bounded (see \cite[Section 2]{DePhilLambPierreVel18}):
\begin{equation}\label{wbd}
\|w_E\|_{L^\infty} \leq C(p) \|h\|_{L^p} |E|^{2/d - 1/p}.
\end{equation}
It also satisfies
\begin{equation}\label{EulLag}
\int_{\R^d} \nabla w_E \cdot \nabla \vhi \, dx = \int_{\R^d}h \vhi \, dx
\end{equation}
for any $\vhi \in \hat H^1_0(E)$, so that
\begin{equation}\label{Diren}
V_{\mathrm{Dir}}(E) = -\frac{1}{2} \int_{\R^d} h w_E \, dx = -\frac{1}{2}\int_{E} h w_E \, dx.
\end{equation}

\begin{proposition}
	The perimeter term $P_{\phi}$ and perturbation term $V_{\mathrm{Dir}}$ satisfy (S1) and (S2).
\end{proposition}

\begin{proof}
Regarding (S1), by combining \eqref{Diren} and the Hölder inequality, for $r>0$ we have
\[
\lt|V_{\mathrm{Dir}}(B_r) \rt|\leq \int_{B_r} h w_{B_r} \leq \|h\|_{L^p(B_r)} \|w_{B_r}\|_{L^{p'}(E)},
\]
where $p' = p/(p-1)$. Thus by \eqref{wbd}
\[
\lt|V_{\mathrm{Dir}}(B_r) \rt| \leq C \|h\|_{L^p}^2 |B_r|^{1/p' +2/d -1/p},
\]
so that $V_{\mathrm{Dir}}(B_r) \to 0$ as $r\to 0$, proving \ref{vanloc}. Next, notice that if $A, B \subset \R^d$ have finite volume and are such that $\mathrm{dist}(A, B) >0$, then $V_{\mathrm{Dir}}(A \cup B) = V_{\mathrm{Dir}}(A) + V_{\mathrm{Dir}}(B)$. Thus \ref{decker} holds. \medskip

We now move to (S2). Regarding \ref{bdEbdP}, given $E \subset \R^d$ we proceed as for \ref{vanloc} and we notice that
\[
V_{\mathrm{Dir}}(E) \geq -C \|h\|_{L^p}^2  |E|^{1/p' +2/d -1/p}.
\]
Therefore,  given $(E_n)_n$ with $\sup_n|E_n|< \infty$, $\sup P_{\phi}(E_n) = \infty$ implies $\sup_n \mathcal{E}(E_n) = \infty$, proving \ref{bdEbdP} by contraposition. Assumption \ref{semconloc} follows from \cite[Remark 2.3]{DePhilLambPierreVel18}. We conclude by proving the weak superaddivity assumption \ref{wksupadd}. Given $m>0$, we fix $E$ with $|E| \leq m$ and $I\geq 1$ disjoints balls$ B_R(x^1), \dots, B_R(x^I)$  such that $\min_{i\neq j} \mathrm{dist}(B^i, B^j) \geq 5R$. We first notice that
\begin{equation}\label{firstnotdir}
\sum_{i=1}^I V_{\mathrm{Dir}}(E \cap B_R(x^i)) = V(E \cap \mathcal{B}), \quad \text{where} \quad \mathcal{B} =  \bigcup_{i=1}^I B_R(x^i).
\end{equation}
Let us denote $B^i = B_{R/2}(x^i)$ for $1 \leq i \leq I$. There exists a mollifier $\psi \in C^{\infty}_c(\R^d, [0,1])$ satisfying 
\[ 
\psi =1 \ \text{ on } \ \mathcal{B} = \bigcup_{i=1}^I B^i, \ \psi = 0 \ \text{ on } \ \bigcap_{i=1}^I B^c_{R}(x_i) \ \text{ and } \ \|\nabla \psi\|_{L^\infty} \leq \frac{C}{R}.  
\]
Noticing that $\psi w_E \in \hat H^1_0(E \cap \mathcal{B})$, we can use it as a candidate for the minimization in $V_{\mathrm{Dir}}(E \cap \mathcal{B})$. We obtain
\[
V_{\mathrm{Dir}}(E \cap \mathcal{B}) -V_{\mathrm{Dir}}(E) \leq \frac{1}{2} \int_{\R^d} |\nabla (\psi w_E)|^2dx + \frac{1}{2}\int_{\R^d} h w_E (1 - 2\psi)dx.
\] 
Notice that $|\nabla (\psi w_E)|^2 = \nabla w_E \cdot \nabla (\psi^2 w_E) + w_E |\nabla \psi|^2$. Applying \eqref{EulLag} with $\vhi = \psi^2 w_E$ yields
\begin{align*}
	 V_{\mathrm{Dir}}(E \cap \mathcal{B}) - V_{\mathrm{Dir}}(E) &\leq \frac{1}{2}\int_{\R^d} w_E^2 |\nabla \psi|^2 dx + \frac{1}{2}\int_{\R^d} w_E h \psi^2 dx  + \frac{1}{2} \int_{\R^d} w_E h (1- 2 \psi) dx\\
	&= \frac{1}{2}\int_{\R^d} w_E^2 |\nabla \psi|^2 dx + \frac{1}{2} \int_{\R^d} w_E h (1- \psi)^2 dx\\
	&\leq \frac{C}{R^2} + \sum_{i=1}^I \int_{(B^i)^c} |w_E h| dx.
\end{align*}
By the Hölder inequality we estimate that for every $1 \leq i \leq I$,
\[
\int_{(B^i)^c} |w_E h| \leq \lt( \int_{E \cap (B^i)^c}|w_E|^{p'} \rt)^{1/p'}\|h \|_{L^p((B^i)^c)} \leq 
	\|w_E|_{L^\infty}^{p'} |E|^{1/p'}  \|h\|_{L^p((B^i)^c)}.
\]
Recall that for $1 \leq i \leq I$, $B^i = B_{R/2}(x^i)$. Thus if we denote
\[
\eta_2(R) = \frac{C}{R^2} +  C \sum_{i=1}^I \|h\|_{L^p((B^i)^c)},
\]
we have that $\eta_2(R) \to 0$ as $R \to \infty$. Thus \ref{wksupadd} holds, as
\[
\sum_{i=1}^I V_{\mathrm{Dir}}(E \cap B^i) = V_{\mathrm{Dir}}(E \cap \mathcal{B}) \leq  V_{\mathrm{Dir}}(E) + \eta_2(R).
\]

\end{proof}

\bibliographystyle{acm}
\bibliography{biblio_PV_generalized}

\begin{thebibliography}{10}

\bibitem{AntNarPoz22}
{\sc Antonelli, G., Nardulli, S., and Pozzetta, M.}
\newblock The isoperimetric problem via direct method in noncompact metric
  measure spaces with lower ricci bounds.
\newblock {\em ESAIM: COCV 28\/} (2022), 57.

\bibitem{BonaCristo14}
{\sc Bonacini, M., and Cristoferi, R.}
\newblock Local and global minimality results for a nonlocal isoperimetric
  problem on $\mathbb{R}^{N}$.
\newblock {\em SIAM J. Math. Anal. 46}, 4 (2014), 2310--2349.

\bibitem{CafRoqSav}
{\sc Caffarelli, L., Roquejoffre, J.-M., and Savin, O.}
\newblock Nonlocal minimal surfaces.
\newblock {\em Comm. Pure Appl. Math. 63}, 9 (2010), 1111--1144.

\bibitem{CanGol22}
{\sc Candau-Tilh, J., and Goldman, M.}
\newblock Existence and stability results for an isoperimetric problem with a
  non-local interaction of {W}asserstein type.
\newblock {\em ESAIM Control Optim. Calc. Var. 28\/} (2022), 37.

\bibitem{CesaNov17}
{\sc Cesaroni, A., and Novaga, M.}
\newblock Volume constrained minimizers of the fractional perimeter with a
  potential energy.
\newblock {\em Discrete Contin. Dyn. Syst. Ser. S 10}, 4, 715--727.

\bibitem{CinPra17}
{\sc Cinti, E., and Pratelli, A.}
\newblock The $\varepsilon$-$\varepsilon^{\beta}$ property, the boundedness of
  isoperimetric sets in $\mathbb{R}^{N}$ with density, and some applications.
\newblock {\em J. Reine Angew. Math. 2017}, 728 (2017), 65--103.

\bibitem{DePhilLambPierreVel18}
{\sc {De Philippis}, G., Lamboley, J., Pierre, M., and Velichkov, B.}
\newblock Regularity of minimizers of shape optimization problems involving
  perimeter.
\newblock {\em J. Math. Pures Appl. 109\/} (2018), 147--181.

\bibitem{DeRNeu23}
{\sc {De Rosa}, A., and Neumayer, R.}
\newblock Local minimizers of the anisotropic isoperimetric problem on closed
  \kern\fontdimen2\font manifolds.
\newblock {\em arXiv preprint arXiv:2308.04565\/} (2023).

\bibitem{DiCNovRufVal}
{\sc Di~Castro, A., Novaga, M., Ruffini, B., and Valdinoci, E.}
\newblock Nonlocal quantitative isoperimetric inequalities.
\newblock {\em Calc. Var. Partial Differ. Equ. 54}, 3 (2015), 2421--2464.

\bibitem{FigFus15}
{\sc Figalli, A., Fusco, N., Maggi, F., Millot, V., and Morini, M.}
\newblock Isoperimetry and stability properties of balls with respect to
  nonlocal energies.
\newblock {\em Comm. Math. Phys. 336\/} (2015), 441--507.

\bibitem{FigMag11}
{\sc Figalli, A., and Maggi, F.}
\newblock On the shape of liquid drops and crystals in the small mass regime.
\newblock {\em Arch. Ration. Mech. Anal. 201\/} (2011), 143--207.

\bibitem{Gam30}
{\sc Gamow, G.}
\newblock Mass defect curve and nuclear constitution.
\newblock {\em Proc. R. Soc. Lond. A 126}, 803 (1930), 632--544.

\bibitem{GolMerPeg22}
{\sc Goldman, M., Merlet, B., and Pegon, M.}
\newblock Uniform ${C}$$^{1,\alpha}$-regularity for almost-minimizers of some
  nonlocal perturbations of the perimeter.
\newblock {\em arXiv preprint arXiv:2209.11006\/} (2022).

\bibitem{GolNov12}
{\sc Goldman, M., and Novaga, M.}
\newblock Volume-constrained minimizers for the prescribed curvature problem in
  periodic media.
\newblock {\em Calc. Var. PDE 44\/} (2012), 297--318.

\bibitem{GolNovRuf22}
{\sc Goldman, M., Novaga, M., and Ruffini, B.}
\newblock Reifenberg flatness for almost minimizers of the perimeter under
  minimal assumptions.
\newblock {\em Proc. Amer. Math. Soc. 150\/} (2022), 1153--1165.

\bibitem{KnMuNov16}
{\sc Knupfer, H., Muratov, C.~B., and Novaga, M.}
\newblock Low density phases in a uniformly charged liquid.
\newblock {\em Comm. Math. Phys. 345}, 1 (2016), 141--183.

\bibitem{KnMuII13}
{\sc Knüpfer, H., and Muratov, C.~B.}
\newblock On an isoperimetric problem with a competing nonlocal term {I}{I}:
  The general case.
\newblock {\em Comm. Pure Appl. Math. 67}, 12 (2013), 1974--1994.

\bibitem{Lio84}
{\sc Lions, P.~L.}
\newblock The concentration-compactness principle in the calculus of
  variations. {T}he locally compact case, part 1.
\newblock {\em Ann. Inst. H. Poincar\'e Anal. Non Lin\'eaire 1}, 2 (1984),
  109--145.

\bibitem{Maggi12}
{\sc Maggi, F.}
\newblock {\em Sets of Finite Perimeter and Geometric Variational Problems: An
  Introduction to Geometric Measure Theory}.
\newblock Cambridge Studies in Advanced Mathematics. Cambridge University
  Press, 2012.

\bibitem{MelWu22}
{\sc Mellet, A., and Wu, Y.}
\newblock An isoperimetric problem with a competing nonlocal singular term.
\newblock {\em Calc. Var. PDE 60\/} (2021).

\bibitem{NovOno22}
{\sc Novaga, M., and Onoue, F.}
\newblock Existence of minimizers for a generalized liquid drop model with
  fractional perimeter.
\newblock {\em Nonlinear Anal. 224\/} (2022), 113078.

\bibitem{NovPaoStepTor21}
{\sc Novaga, M., Paolini, E., Stepanov, E.~O., and Tortorelli, V.~M.}
\newblock Isoperimetric clusters in homogeneous spaces via concentration
  compactness.
\newblock {\em J. Geom. Anal. 32\/} (2021).

\bibitem{Rig2000}
{\sc Rigot, S.}
\newblock Ensembles quasi-minimaux avec contrainte de volume et
  rectifiabilit\'e uniforme. (quasi-minimal sets with a volume constraint and
  uniform rectifiability).
\newblock {\em M\'emoires de la Soci\'et\'e Math\'ematique de France. Nouvelle
  S\'erie 82\/} (01 2000).

\end{thebibliography}

\end{document}